\documentclass{amsart}

\usepackage[hmargin=2.5cm,top=1.5cm,bottom=1.5cm,includehead]{geometry}

\usepackage{amsmath,mathrsfs}
\usepackage{amssymb, amsmath}
\usepackage{graphicx}
\usepackage{enumerate,color,xcolor}

\usepackage[normalem]{ulem}

\newtheorem{thm}{Theorem}[section]

\newtheorem{rmk}{Remark}[section]

\newtheorem{prop}{Proposition}[section]

\newcommand{\vv}[1]{\boldsymbol{#1}}

\numberwithin{equation}{section}

\newcommand{\beq}{\begin{equation}}
	\newcommand{\eeq}{\end{equation}}
\newcommand{\ben}{\begin{eqnarray}}
	\newcommand{\een}{\end{eqnarray}}
\newcommand{\beno}{\begin{eqnarray*}}
	\newcommand{\eeno}{\end{eqnarray*}}
\let\f=\frac

\newcommand{\bit}{\begin{itemize}}
	\newcommand{\eit}{\end{itemize}}

\newcommand{\N}{{\mathbb N}}
\newcommand{\Z}{{\mathbb Z}}
\newcommand{\R}{{\mathbb R}}
\newcommand{\T}{{\mathbb T}}
\newcommand{\pa}{\partial}
\newcommand{\eps}{\varepsilon}
\newcommand{\vf}{\varphi_{\geq0}}
\newcommand{\ov}{\overline}

\newcommand{\ba}{\begin{aligned}}
	\newcommand{\ea}{\end{aligned}}
\def\na{\nabla}

\let\wt=\widetilde
\let\wh=\widehat
\def\cI{{\mathcal I}}

\def\cN{{\mathcal N}}
\def\sD{{\mathscr D}}
\def\sL{{\mathscr L}}
\def\sN{{\mathscr N}}
\def\sR{{\mathscr R}}
\def\bI{{\mathbf I}}
\def\bN{{\mathbf N}}
\def\virgp{\raise 2pt\hbox{,}}
\def\cdotpv{\raise 2pt\hbox{;}}

\def\<{\langle}
\def\>{\rangle}

\def\gs{\gtrsim}
\def\ls{\lesssim}
\newcommand{\rr}[1]{\left( #1 \right)}
\newcommand{\nr}[1]{\left| #1 \right|}
\newcommand{\nnr}[1]{\left\| #1 \right\|}
\newcommand{\LLT}[1]{\left\| #1 \right\|_{L^2_{x,v}}}
\newcommand{\LLTX}[1]{\left\| #1 \right\|_{L^2_{x}}}
\newcommand{\LLTXT}[1]{\left\| #1 \right\|_{L^2([0,T^*];L^2_{x})}}

\def\vect{\pa_v^t}
\def\smc{{\mathsf b}}

\def\ent{e^{\nu t}}
\def\entau{e^{\nu\tau}}
\def\emnt{e^{-\nu t}}
\def\emntau{e^{-\nu\tau}}
\def\tent{\f{1-\emnt}{\nu}}
\def\tentau{\f{1-\emntau}{\nu}}

\begin{document}

\title[Weak collision effect on nonlinear Landau damping]{Weak collision effect on nonlinear Landau damping for the Vlasov-Poisson-Fokker-Planck system}

\author{Yue Luo}
\address[Yue Luo]{Department of Mathematical Sciences, Tsinghua University, Beijng, 100084, P.R. China.}
\email{luo-y21@mails.tsinghua.edu.cn}

\begin{abstract}
	We investigate the impact of weak collisions on Landau damping in the Vlasov-Poisson-Fokker-Planck system on a torus, specifically focusing on its proximity to a Maxwellian distribution. In the case where the Gevrey index satisfies $\frac{1}{s}\leq3$, we establish the global stability and enhanced dissipation of small initial data, which remain unaffected by the small diffusion coefficient $\nu$. For Gevrey index $\frac{1}{s}>3$, we prove the global stability and enhanced dissipation of initial data, whose size is on the order of $O(\nu^\f{1-3s}{3-3s})$. Our analysis provides insights into the effects of phase mixing, enhanced dissipation, and plasma echoes.	
\end{abstract}

\maketitle





\setcounter{tocdepth}{1}
\section{Introduction}

We examine the following single-species Vlasov-Poisson-Fokker-Plank system with a neutralizing background on $\mathbb{T}_x^d\times\R^d_v$ (where $\mathbb{T}_x$ has a normalized length of $2\pi$) with a small diffusion coefficient $\nu$. This system models the evolution of the distribution function of electrons influenced by a self-consistent electric field and a Brownian force.
\ben\label{VPF}
\left\{\ba
&\pa_tF+v\cdot\na_xF+E\cdot\na_vF=\nu\rr{\Delta_vF+\na_v\cdot(vF)}, \quad(x,v)\in\T^d\times\R^d,\\
&E(t,x)=-\na_x(-\Delta_x)^{-1}\rr{\varrho(t,x)-\f{1}{|\T^d|}\int_{\T^d}\varrho dx}, \ \varrho(t,x)=\int_{\R^d} Fdv.
\ea\right.
\een
The function $F(t,x,v)$ represents the distribution of electrons, indicating the number of electrons at position x with velocity v. The spatial density of electrons is given by $\varrho(t,x)$. The electric field E(t,x) arises from the Coulomb interaction among the electrons. The Brownian force, which models the random motion of electrons due to collisions, is described by the Fokker-Planck operator $\Delta_vF+\na_v\cdot(vF)$. Here, the diffusion coefficient $\nu$ signifies the frequency of collisions.

By performing a direct computation, it can be verified that the system \eqref{VPF} satisfies the conservation of mass:
\[\f{d}{dt}\iint F(t)dvdx=0.\]
We normalize the conserved quantity $\iint F(0)dvdx=1$ and define the following global Maxwellian (with unit mass):
\[\mu(v)=\f{1}{|\T^d|}\f{1}{(2\pi)^\f d2}e^{-\f{|v|^2}{2}}.\]
Then \eqref{VPF} satisfies the following  equality
\ben\label{dee}f{d}{dt}\rr{\iint F\log \f F\mu dvdx+\f12\int E^2dx}+4\nu\iint\nr{\na_v\sqrt{\f F\mu}}^2\mu dvdx=0,\een
which indicates the dissipation of energy-entropy due to the Brownian force. Moreover, \eqref{dee} shows that \eqref{VPF} possesses a unique equilibrium $\mu$ which minimizes the energy-entropy. Consequently, we will examine \eqref{VPF} in the vicinity of the global Maxwellian $\mu$.

\subsection{Background and short review of the recent work} 
The collision of plasma is usually very weak and can be neglected in many settings. For collisionless plasma, the most famous phenomenon is Landau damping. This nondissipative relaxation effect was discovered by Lev Landau in 1946 when he proposed the linear stability of certain nontrivial homogeneous equilibria, particularly the Maxwellian distribution, for the Vlasov-Poisson equation. At the linear level, the key mechanism for Landau damping is phase mixing, which is the rapid homogenization of density caused by the transfer of spatial information to small scales in velocity by the transport operator and the velocity averaging effect due to integration in $v$ for $\varrho$, leading to a rapid decay of the electric field. It has been shown in \cite{Penrose} that for homogeneous background solutions satisfying the Penrose condition, the Volterra equation for density is stable, thus guaranteeing linear stability.

In the realm of the full nonlinear equation, a significant resonance phenomenon known as the plasma echo emerges as the primary hurdle in achieving nonlinear Landau damping. This effect arises from nonlinear interactions that excite modes which are un-mixing in phase space, leading to a temporary amplification of the electric field. This transient growth can subsequently trigger further oscillations, creating a cascade effect. The presence of plasma echoes necessitates stringent regularity conditions for perturbations.The initial validation of nonlinear Landau damping was accomplished by Mouhot and Villani \cite{MV}, who demonstrated the global stability of analytic and nearly analytic perturbations and hypothesized a Gevrey threshold of 3. Subsequently, Bedrossian et al. \cite{BMM} provided a more straightforward proof, reducing the required regularity to the Gevrey threshold of $\f1s<3$ . More recently, Ionescu et al. \cite{waveop} confirmed the Gevrey threshold of 3 and established the existence, injectivity, and Lipschitz estimates of the scattering operators. For a mathematical study of plasma echo and instability of Sobolev perturbations see \cite{echoes}. For a comprehensive overview of the mathematical theory of Landau damping, we recommend consulting \cite{review} and its cited references.

When weak collisions are taken into account, the interaction between the transport operator and dissipation gives rise to new phenomena. This can be observed from the inhomogeneous Fokker-Planck equation $\pa_tF+v\cdot\na_xF=\nu\rr{\Delta_vF+\na_v\cdot(vF)}$. The explicit solution is 
\[\wh{F}(t,k,\eta)=\wh{F}(0,k,e^{-\nu t}\eta-\f{e^{-\nu t}-1}{\nu}k)\exp{\rr{-\int_0^t|e^{-\nu \tau}\eta-\f{e^{-\nu \tau}-1}{\nu}k|^2d\tau}}.\]
Upon careful analysis, one can see that the spatial zero mode tends towards thermodynamic equilibrium at a time scale of $\mathcal{O}(\nu^{-1})$, aligning with the thermalization time scale of the homogeneous Fokker-Planck equation. In contrast, the nonzero spatial modes decay at a shorter time scale of $\mathcal{O}(\nu^{-\f13})$. This phenomenon, known as enhanced dissipation, results from the transfer of spatial information to high frequencies in velocity, where dissipation becomes increasingly dominant. We refer to \cite{taylor} and the references therein for more discussions on enhanced dissipation.

Given that collisions are inherent in plasma, the impact of weak collisions on nonlinear Landau damping has been a longstanding query.  It is plausible to anticipate that the dissipation effect could suppress the nonlinear resonance effect.  The first work on it was by \cite{Bj}, who demonstrated global stability and enhanced dissipation for Sobolev perturbations of Maxwellian of size $\mathcal{O}(\nu^{\frac{1}{3}})$. Subsequently, Chaturvedi et al. \cite{CLT} established a comparable result for a more complex Vlasov-Poisson-Landau system. The stability threshold of $\nu^{\frac{1}{3}}$ arises from the following argument: Landau damping ensures stability of Sobolev perturbations of size $\mathcal{O}(\eps)$ over the time scale of $\mathcal{O}(\eps^{-1})$, whereas enhanced dissipation kicks in and dominates at the time scale of $\mathcal{O}(\nu^{-\f13})$. For $\eps^{-1}\gs\nu^{-\f13}$, the interplay of these two mechanisms should ensure global stability.

We aim to gain a deeper understanding of the global behavior of small perturbations of the Maxwellian distribution within the framework of the Vlasov-Poisson-Fokker-Planck system. Specifically we are curious about the stability threshold in Gevrey-$\f1s$ class for $s<\f13$ and the global stability for small perturbations of size uniform with respect to $\nu$ in Gevrey-$\f13$ class. Both questions necessitate a better understanding of the interplay between echo chains and enhanced dissipation. It is important to note that, due to technical limitations, our analysis is restricted to the Vlasov-Poisson-Fokker-Planck system. The questions regarding the more physically relevant Vlasov-Poisson-Landau system remain open and require further investigation.

\subsection{Notations}
For $f(x,v)$ defined on $\T^d\times\R^d$ we define the Fourier transform $\wh{f}(k,\eta)$, where $(k,\eta)\in \Z^d\times\R^d$ via
\[\wh{f}(k,\eta)=\f{1}{(2\pi)^d}\int_{\T^d\times\R^d}f(x,v)e^{-ikx-i\eta v}dxdv.\]
For $f(x,v)$ we define the following components:
\[f_0:=\f{1}{(2\pi)^d}\int_{\T^d}f(x,v)dx,\quad f_{\neq}:=f-f_0.\]
Here, $f_0$ denotes the zero spatial Fourier mode of $f$, while $f_{\neq}$ represents the nonzero spatial Fourier mode of $f$. Weighted Sobolev norms are given as 
\[\|f\|_{H^\sigma_q}^2=\int\<v\>^q|\<\na\>^\sigma f|^2dxdv.\]
We remark that 
\[\|f\|_{H^\sigma_q}^2\sim_{\sigma,q}\int|\<\na\>^\sigma \<v\>^qf|^2dxdv.\]

We denote $\<x\>=\rr{1+|x|^2}^\f12$. For $(k,\eta)\in\T^d\times\R^d$, we denote $|k,\eta|=\rr{|k|^2+|\eta|^2}^\f12$, $\<k,\eta\>=\<|k,\eta|\>$. We use the notation $f\ls g$ when there exists a constant $C>0$ such that $f\leq Cg$ (we analogously define $f\gs g$). Similarly, we use the notation $f\sim g$ when there exists $C>0$ such that $C^{-1}g\leq f\leq Cg$.

\subsection{Reformulation of the problem} In this subsection we introduce, in a similar way as \cite{Bj}, a change of variables that reduces \eqref{VPF} to a system whose structure is more transparent. We set $F(t)=\mu+h(t)$ and $\varrho(t)=1+\rho$. Then $h(t)$ satisfies the following equation:
\ben\label{1}\left\{\ba
&\pa_th+v\cdot\na_xh+E\cdot\na_v\mu+E\cdot\na_vh=\nu\rr{\Delta_vh+\na_v\cdot\rr{vh}},\\
&E(t,x)=-\na_x(-\Delta_x)^{-1}\rr{\rho(t,x)}.
\ea\right.\een
We recall that $\iint h(t,x,v)dxdv=0$, which implies that $\wh{\rho}(t,0)=0$.  By performing Fourier transform with respect to $x$ and $v$ variables, we get that
\[\ba
\left\{\ba&\pa_t\wh{h}(t,k,\eta)+(\nu\eta-k)\cdot\na_\eta\wh{h}(t,k,\eta)+\wh{\rho}(t,k)\f{k\cdot\eta}{|k|^2}\wh{\mu}(\eta)+\sum_\ell\wh{\rho}(t,\ell)\f{\ell\cdot\eta}{|\ell|^2}\wh{h}(t,k-\ell,\eta)=-\nu|\eta|^2\wh{h}(t,k,\eta),\\
&\wh{\rho}(t,k)=\wh{h}(t,k,kt).
\ea\right.
\ea\]
We define the following coordinate shift:
\[\bar{\eta}(t,k,\eta):=\ent\eta-k\f{\ent-1}{\nu}.\]
If we set \ben\label{defoff}\wh{f}(t,k,\eta):=\wh{h}(t,k,\bar{\eta}(t,k,\eta)),\een then $f$ satisfies
\[
\left\{\ba&\pa_t\wh{f}(t,k,\eta)+\wh{\rho}(t,k)\f{k\cdot\bar{\eta}}{|k|^2}\wh{\mu}(\bar{\eta})+\sum_\ell\wh{\rho}(t,\ell)\f{\ell\cdot\bar{\eta}}{|\ell|^2}\wh{f}(t,k-\ell,\eta-\ell\tent)=-\nu|\bar{\eta}|^2\wh{f}(t,k,\eta),\\
&\wh{\rho}(t,k)=\wh{f}(t,k,k\tent).
\ea\right.
\]
Following \cite{MV,BMM,Bj}, to derive an equation for density, we set  
\[S(t,\tau,k,\eta):=\exp\rr{-\nu\int_\tau^t\nr{\bar{\eta}(\tau',k,\eta)}^2d\tau'}.\]
Thanks to  Duhamel's principle, we have
\[\ba
\wh{f}(t,k,\eta)=&\wh{f_{in}}(k,\eta)S(t,0,k,\eta)-\int_0^t\wh{\rho}(\tau,k)\f{k\cdot\bar{\eta}(\tau,k,\eta)}{|k|^2}\wh{\mu}(\bar{\eta}(\tau,k,\eta))S(t,\tau,k,\eta)d\tau\\
&-\int_0^t\sum_\ell\wh{\rho}(\tau,\ell)\f{\ell\cdot\bar{\eta}(\tau,k,\eta)}{|\ell|^2}\wh{f}(\tau,k-\ell,\eta-\ell\tentau)S(t,\tau,k,\eta)d\tau.
\ea\]
By setting $\eta:=k\tent$ we get
\ben\label{rho}
\wh{\rho}(t,k)=\wh{h_{in}}(k,k\tent)S(t,k)-\int_0^t\wh{\rho}(\tau,k)\f{1-e^{-\nu(t-\tau)}}{\nu}\wh{\mu}(k\f{1-e^{-\nu(t-\tau)}}{\nu})S(t-\tau,k)d\tau\notag\\
-\sum_\ell\int_0^t\wh{\rho}(\tau,\ell)\f{\ell\cdot k}{|\ell|^2}\f{1-e^{-\nu(t-\tau)}}{\nu}\wh{f}(\tau,k-\ell,k\tent-\ell\tentau)S(t-\tau,k)d\tau,
\een
where we use the shorthand, i.e.,
\[S(t-\tau,k):=S(t,\tau,k,k\tent)=\exp\rr{-\nu\int_0^{t-\tau}\nr{k\f{1-e^{-\nu \tau'}}{\nu} }^2d\tau'}.\]

\subsection{Main results} Our main result can be summarized as follows:
\begin{thm}\label{T1}
	Let $m\geq\f d2+2$ be an integer. There exists $\nu_0(d,m)>0, \delta_2(d,m)>0$ such that
	\begin{itemize}
		\item[(1)] For $\f13\leq s\leq1$, for any $\lambda_1>\lambda_\infty>0$, there exists $\eps_0(s,\lambda_1,\lambda_\infty,d,m)>0$ such that for any $\nu<\nu_0$, for any mean-zero initial data $h_{in}$ satisfying
		\[\LLT{\<v\>^m e^{\lambda_1\<\na_{x,v}\>^s}h_{in}}=\eps<\eps_0(s,\lambda_1,\lambda_\infty,d,m),\]
		\eqref{1} has a unique global classical solution satisfying
		\[\ba
		&\LLT{\<v\>^m e^{\lambda_\infty\<\na_{x,v}\>^s}f_{\neq}(t)}\ls_{m,d}\eps e^{-\delta_2\nu^\f13t},\\
		&\nnr{e^{\lambda_\infty\<\emnt\na_v\>^s}h_0(t)}_{L^2_m}\ls_{m,d}\eps\emnt,\\
		&\nr{\wh{\rho}(t,k)}\ls_{m,d,N}\eps\<k\>^{-N}\<t\>^{-N}\quad \forall N\geq0.
		\ea\]
		\item[(2)] For $0<s<\f13$, for any $\lambda_1>\lambda_\infty>0$, there exists $c(s,\lambda_1,\lambda_\infty,d,m)>0$ such that for any $\nu<\nu_0$, for any mean-zero initial data $h_{in}$ satisfying
		\[\LLT{\<v\>^m e^{\lambda_1\<\na_{x,v}\>^s}h_{in}}=\eps<c(s,\lambda_1,\lambda_\infty,d,m)\nu^{\f{1-3s}{3-3s}},\]
		\eqref{1} has a unique global classical solution satisfying
		\[\ba
		&\LLT{\<v\>^m e^{\lambda_\infty\<\na_{x,v}\>^s}f_{\neq}(t)}\ls_{m,d}\eps e^{-\delta_2\nu^\f13t},\\
		&\nnr{e^{\lambda_\infty\<\emnt\na_v\>^s}h_0(t)}_{L^2_m}\ls_{m,d}\eps\emnt,\\
		&\nr{\wh{\rho}(t,k)}\ls_{m,d,N}\eps\<k\>^{-N}\<t\>^{-N}\quad \forall N\geq0.
		\ea\]
	\end{itemize}
Here $f$ is defined by \eqref{defoff}.
\end{thm}
\begin{rmk}
	After completing our initial work, we discovered an independent study by Bedrossian, Zhao, and Zi (see \cite{BZZ}) through the schedule of a workshop held at the Tsinghua Sanya International Mathematics Forum (TSIMF) from January 15 to 19, 2024. Inspired by their abstract, we improved the index of the stability threshold in Theorem \ref{T1}(2) from $\f{1-3s}{3-4s}$ to $\f{1-3s}{3-3s}$. However, we are uncertain whether this index is optimal. Specifically, we have not been able to prove the existence of unstable solutions for initial data above this threshold.
\end{rmk}

\subsection{Heuristic}
In this subsection, we use a heuristic, similar to those in \cite{echoes} and \cite{review}, to illustrate how enhanced dissipation might influence plasma echoes. For simplicity, we replace the Fokker-Planck operator with the Laplacian $\Delta_v$. Given that plasma echoes are primarily a nonlinear phenomenon, we omit the linear electric field term. This leads to the following formal expression for the density:
\[\wh{\rho}(t,k)=\wh{h_{in}}(k,kt)e^{-\f13\nu|k|^2t^3}-\sum_\ell\int_0^t\wh{\rho}(\tau,\ell)\f{\ell\cdot k(t-\tau)}{|\ell|^2}\wh{f}(\tau,k-\ell,kt-\ell\tau)e^{-\f13\nu|k|^2(t-\tau)^3}d\tau.\]
We reduce the system to 1D and assume $k>0$. We assume $f$ is of size $\nu^a\ (a<\f13)$ and has sufficiently high regularity, namely, $|\wh{f_{\neq}}(t,k,\eta)|\ls\nu^a\<k\>^{-100}\<\eta\>^{-100}e^{-\nu^\f13t}$. The time decay factor $e^{-\nu^\f13t}$ arises due to enhanced dissipation. We focus on the worst-case scenario, namely, $\ell=k\pm1$,
\[\nr{\wh{\rho}(t,k)}\ls\sum_{\ell=k\pm1}\nu^a\int_0^t\nr{\wh{\rho}(\tau,\ell)}(t-\tau)\<kt-\ell\tau\>^{-100}e^{-\nu^\f13\tau}e^{-\f13\nu|k|^2(t-\tau)^3}d\tau.\]
For $\ell=k-1$ we notice that $\<kt-\ell\tau\>^{-100}\leq\<t\>^{-100}$. In other words, the impact of the $k-1$ mode on the $k$ mode is very weak. Thus we only consider the case $\ell=k+1$. Formally we have
\[\nr{\wh{\rho}(t,k)}\ls\nu^a\nr{\wh{\rho}(\f{kt}{k+1},k+1)}\f{t}{k+1}e^{-\nu^\f13\f{kt}{k+1}}\int_0^t\<kt-(k+1)\tau\>^{-100}d\tau\ls\f{\nu^at}{(k+1)^2}e^{-\nu^\f13\f{kt}{k+1}}\nr{\wh{\rho}(\f{kt}{k+1},k+1)}.\]
It shows that the $k+1$ mode at time $\f{kt}{k+1}$ has a strong impact on the $k$ mode at time $t$. The $k$ mode may excite subsequent modes, thereby creating a cascade effect. We rewrite this plasma echo chain in terms of the distribution function (noting that $\eta=kt$):
\[\nr{\wh{f}(\f{\eta}{k_1},k_1,\eta)}\leq\f{\nu^a\eta}{(k_1+1)^3}e^{-\nu^\f13\f{\eta}{k_1+1}}\nr{\wh{f}(\f{\eta}{k_1+1},k_1+1,\eta)}\leq\dots\leq\rr{\prod_{k=k_1+1}^{k_2} \f{\nu^a\eta}{k^3}e^{-\nu^\f13\f{\eta}{k}}} \nr{\wh{f}(\f{\eta}{k_2},k_2,\eta)}.\]

In the following we will estimate $\sup_{k_1,k_2\in\N}\prod_{k=k_1}^{k_2} \f{\nu^a\eta}{k^3}e^{-\f{\nu^\f13\eta}{k}}$, which represents the maximal impact of echo chains. We will consider separately three cases:
\begin{itemize}
    \item[(i)] $\eta\leq3\nu^{-\f13}$. In this case we have
    \[\f{\nu^a\eta}{k^3}e^{-3}\leq\f{\nu^a\eta}{k^3}e^{-\f{\nu^\f13\eta}{k}}\leq\f{\nu^a\eta}{k^3}.\]
    Thus by Stirling's formula it holds that
    \[e^{3\f{(\nu^a\eta)^\f13}{e}}\rr{\f{e^3}{\nu^a\eta}}^\f12\ls\prod_{k=1}^{[\f{(\nu^a\eta)^\f13}{e}]}\f{\nu^a\eta}{k^3}e^{-3}\leq\sup_{k_1,k_2\in\N}\prod_{k=k_1}^{k_2} \f{\nu^a\eta}{k^3}e^{-\f{\nu^\f13\eta}{k}}\leq\prod_{k=1}^{[(\nu^a\eta)^\f13]}\f{\nu^a\eta}{k^3}\ls e^{3(\nu^a\eta)^\f13}\f{1}{(\nu^a\eta)^\f12}.\]
    \item[(ii)] $3\nu^{-\f13}\leq\eta\leq\rr{\f{3}{e}}^\f32\nu^{-\f{1-a}{2}}$. We observe that the function $\psi(k):=\f{\nu^a\eta}{k^3}e^{-\f{\nu^\f13\eta}{k}}$ attains its maximum value at $k=\f13\nu^\f13\eta$. For $k\geq\f13\nu^\f13\eta$,
    \[\f{\nu^a\eta}{k^3}e^{-3}\leq\f{\nu^a\eta}{k^3}e^{-\f{\nu^\f13\eta}{k}}\leq\f{\nu^a\eta}{k^3}.\]
    For $k\leq\f13\nu^\f13\eta$, 
    \[\f{\nu^a\eta}{k^3}e^{-\f{\nu^\f13\eta}{k}}\leq\psi(\f13\nu^\f13\eta)=\f{27}{e^3}\nu^{a-1}\eta^{-2}.\]
    Thus we have
    \[\sup_{k_1,k_2\in\N}\prod_{k=k_1}^{k_2} \f{\nu^a\eta}{k^3}e^{-\f{\nu^\f13\eta}{k}}\geq\prod_{k=[\f13\nu^\f13\eta]+1}^{[(\f{\nu^a\eta}{e^3})^\f13]}\f{\nu^a\eta}{k^3}e^{-3}\gs\rr{\f{\nu^\f13\eta}{3(\nu^a\eta)^\f13}}^{\nu^\f13\eta+\f32}e^{\f{3(\nu^a\eta)^\f13}{e}},\]
    and
    \[\sup_{k_1,k_2\in\N}\prod_{k=k_1}^{k_2} \f{\nu^a\eta}{k^3}e^{-\f{\nu^\f13\eta}{k}}\leq\rr{\f{27}{e^3}\nu^{a-1}\eta^{-2}}^{[\f13\nu^\f13\eta]}\prod_{k=[\f13\nu^\f13\eta]+1}^{[(\nu^a\eta)^\f13]}\f{\nu^a\eta}{k^3}\ls\rr{\nu^{1-a}\eta^2}^\f12e^{3(\nu^a\eta)^\f13-2\nu^\f13\eta}.\]
    \item[(iii)] $\eta\geq\rr{\f{3}{e}}^\f32\nu^{-\f{1-a}{2}}$. In this case $\psi(k)\leq\psi(\f13\nu^\f13\eta)=\f{27}{e^3}\nu^{a-1}\eta^2\leq1$. Therefore $\sup_{k_1,k_2\in\N}\prod_{k=k_1}^{k_2} \f{\nu^a\eta}{k^3}e^{-\f{\nu^\f13\eta}{k}}\leq1$.
\end{itemize}

The previous estimate indicates that the plasma echo may induce a growth of $e^{3(\nu^a\eta)^\f13-2\nu^\f13\eta}$ in the frequency space. It can be readily verified that $e^{3(\nu^a\eta)^\f13-2\nu^\f13\eta}\leq e^{3\eta^{\f{1-3a}{3-3a}}}$. To counteract this growth, we will employ a time-dependent Gevrey-$\f1s$ muliplier.

\subsection{Main idea of the proof} We begin by introducing the Fourier multipliers and the energy functional that will be utilized throughout the subsequent analysis. Let $b=s/8$. Motivated by \cite{waveop}, we define the function $\lambda(t,r)$ as follows: \[\lambda(t,r)=\lambda_\infty+\f{\lambda_1-\lambda_\infty}{8}(1+t)^{-b}\<r\>^s+\f{\lambda_1-\lambda_\infty}{8}(1+t\<r\>^{s-1})^{-b}\<r\>^s.\]
By a rather tedious computation one can check (as in \cite{waveop}) that $\lambda(t,r)$ satisfies the following properties:
\begin{itemize}
	\item[(i)] For any $t\geq0$, $x,y\geq0$, the following inequality holds: $\lambda(t,x+y)\leq\lambda(t,x)+\lambda(t,y)$.
	\item[(ii))] For any $t\geq0$, $x\geq y\geq2\<x-y\>$, the following inequality holds:
	$e^{\lambda(t,x)}-e^{\lambda(t,y)}\ls\<y\>^{s-1}e^{\lambda(t,x)}(x-y).$
\end{itemize}

Let $\beta\geq\max\{\f d2+m+3,5\}$. Set $\vect=\bar{\eta}(t,\na_x,\na_v)$. We define the following mutipliers:
\[A^\nu(t,k,\eta)=\<k,\eta\>^\beta e^{\lambda(\tent,|k,\eta|)},\quad A^\nu_c(t,k,\eta)=\<k,\eta\>^{\beta+c} e^{\lambda(\tent,|k,\eta|)}\quad\forall c\geq-\beta,\]
\[\wh{A^\nu_cf}(t,k,\eta)=A^\nu_c(t,k,\eta)\wh{f}(t,k,\eta),\quad\wh{A^\nu_c\rho}(t,k)=A^\nu_c(t,k,k\tent)\wh{\rho}(t,k)\quad\forall c\geq-\beta.\]

As has been done in several previous works on Landau damping \cite{BMM}\cite{Bj}, we apply a bootstrap argument. Let $\delta>\delta_1>\delta_2>0$, $K_{\rho}, K_f, K_{ED}, K_{SH}>0$, $\smc$ be a small constant, $K_1<K_2<K_3<\cdots<K_m$ be large constants which will be determined in the sequel. We introduce the following controls referred to in the sequel as the bootstrap hypotheses: 
\ben\label{Hrho}&&\LLTXT{|\na_x|^\f12A^\nu_\f12\rho(t)e^{\delta\nu^\f13t}\<\tent\>^4}\leq 2K_{\rho}\eps.\\
\label{Topf}&&\sum_{|\alpha|\leq m}\f{e^{-2(|\alpha|+1)\nu t}}{K_{|\alpha|}}\LLT{A^\nu_2(v^\alpha f)(t)}^2\leq (2K_{f})^2\eps^2.\\
\label{Hed}&&\sum_{|\alpha|\leq m}\f{e^{-2(|\alpha|+1)\nu t}}{K_{|\alpha|}}\left(\LLT{\na_x^2A^\nu(v^\alpha f)(t)}^2+\smc\nu^\f23\LLT{\vect \na_xA^\nu(v^\alpha f)(t)}^2\right)\leq (2K_{ED})^2\eps^2e^{-2\delta_1\nu^\f13t}.\\
\label{Hsh}&&\nnr{A^\nu_{-\beta}(t,0,\emnt\na)h_0(t)}_{H^{\beta+1}_m}\leq 2K_{SH}\eps\emnt.\een

We remark that \eqref{Topf} involves the top order derivatives of the distribution function with respect to both the $x$ and $v$ variables, but it does not provide any time decay. In contrast, \eqref{Hed} is designed to capture the enhanced dissipation structure and includes a time decay factor of $e^{-\delta_1\nu^\f13t}$. The additional derivative on $x$ is to restrict $f$ to its spatial nonzero mode $f_{\neq}$. The gain of $\f12$ derivative in \eqref{Hrho} aligns with the regularity improvements observed in velocity averaging lemmas. For more detailed discussions, we refer the reader to \cite{Bj}.

By employing a standard argument, one can establish the local well-posedness of equation \eqref{1} and demonstrate that the estimates \eqref{Hrho}--\eqref{Hsh} hold for a very short time interval. In the subsequent sections, we will utilize a standard continuity argument to prove that, under appropriately chosen constants and for sufficiently small values of $\eps$ and $\nu$, the estimates \eqref{Hrho}--\eqref{Hsh} are valid for all time. More specifically, in Section 2 we will improve \eqref{Hrho}. In Section 3 we will improve \eqref{Topf} by a standard energy method. Section 4 is dedicated to improving \eqref{Hed}. We employ a hypocoercivity energy functional that incorporates an inner product term to generate damping of $\na_xf$, thereby making use of the enhanced dissipation structure. We refer to \cite{taylor} and \cite{CLT} for more discussions on the hypocoercivity method. In Section 6 we will improve \eqref{Hsh}, in which we apply the semigroup method to capture the thermalization effect of the homogeneous Fokker-Plank equation. In Section 7 we will give a detailed proof of Theorem \ref{T1}.

\bigskip

At the end of this subsection, we list several propositions that will be frequently used in the sequel. 

\begin{prop}\label{S}
	\begin{itemize}
		\item[(1)] There exists $\delta'>0$ such that
		\[S(t,k)\leq e^{-\delta'\min\{\nu|k|^2t^3, \f{k^2}{\nu}t\}}.\]
		\item[(2)] For all $p\in(0,1]$ and sufficiently small $\delta$ (depending on $p$),
		\[e^{\delta\nu^{\f13}t}S^p(t-\tau,k)\ls_{p,\delta} e^{\delta\nu^{\f13}\tau}.\]
	\end{itemize}
\end{prop}

\begin{proof}
	See \cite{Bj} Lemma 3.1.
\end{proof}

\begin{prop}\label{nut}
There exists $\nu_0>0$ such that for any $\nu<\nu_0, t\geq0$,
\[\<\nu t\>^{-1}\<\tent\>^{-1}\ls\<t\>^{-1}.\]
\end{prop}

\begin{proof}
	The inequality can be analyzed by considering two separate cases: $\nu t\leq1$ and $\nu t\geq1$. In each of these cases, the inequality holds trivially.
\end{proof}

\section{Estimates of the density}
The goal of this section is to improve the estimate \eqref{Hrho} for the density. 

\subsection{Estimate of the Volterra equation}
We rewrite \eqref{rho} as a Volterra-type equation
\ben\label{Volterra}\wh{\rho}(t,k)=H^\nu(t,k)-\int_0^t\wh{\rho}(\tau,k)K^\nu(t-\tau,k)d\tau,\een
where
\ben\label{Knu} &&K^\nu(t,k):=\tent\wh{\mu}(k\tent)S(t,k);\\
 \label{Hnu}&&H^\nu(t,k):=\wh{h}(0,k,k\tent)S(t,k)-\int_0^t\wh{\rho}(\tau,k)\f{1-e^{-\nu(t-\tau)}}{\nu}\wh{h}(\tau,0,k\f{1-e^{-\nu(t-\tau)}}{\nu})S(t-\tau,k)d\tau\\
&&-\sum_{\ell\neq k}\int_0^t\wh{\rho}(\tau,\ell)\f{\ell\cdot k}{|\ell|^2}\f{1-e^{-\nu(t-\tau)}}{\nu}\wh{f}(\tau,k-\ell,k\tent-\ell\tentau)S(t-\tau,k)d\tau \notag\\&&
:=\wh{\cI}(t,k)+\wh{\cN_{0}}(t,k)+\wh{\cN_{\neq}}(t,k).\notag
\een
Following the previous works \cite{GNR}\cite{review}, we introduce the resolvent $R^\nu(t,k)$  that satisfies
\ben\label{resol}R^\nu(t,k)=K^\nu(t,k)-\int_0^tR^\nu(\tau,k)K^\nu(t-\tau,k)d\tau.\een
Then the  solution of \eqref{Volterra} can be expressed by
\ben\label{newforumladens}\wh{\rho}(t,k)=H^\nu(t,k)-\int_0^t H^\nu(\tau,k)R^\nu(t-\tau,k)d\tau.\een
Next we derive an estimate of the resolvent $R^\nu(t,k)$.
\begin{prop}
	There exists $\bar\lambda>0$ such that the following inequality holds:
	\ben\label{resolvent}\nr{R^\nu(t,k)}\ls\f{1}{|k|}e^{-\bar{\lambda}|k|t}.\een
\end{prop}

To prove this Proposition we introduce the Laplace transform and the inverse Laplace transform (see, for example, Appendix D in \cite{review}),
\ben\wt{f}(z)=\int_0^\infty f(t)e^{-zt}dt,\quad z\in\mathbb{C},\quad f(t)=\f{1}{2\pi }\int_\R \wt{f}(\lambda+i\omega)e^{(\lambda+i\omega)t}d\omega.\een
Next we provide an estimate of the Laplace transform of the Volterra kernel $K^\nu(t,k)$.

\begin{prop}
	There exists $\nu_0>0, \bar{\lambda}>0, \kappa>0$ such that for any $\nu<\nu_0$, 
	\ben\label{Penrose}\inf_{k\in \Z^d\backslash\{0\}}\inf_{\Re z\geq-\bar{\lambda}|k|}\nr{\wt{K^\nu}(z,k)+1}\geq\kappa,\een
	\ben\label{Im}\nr{\wt{K^\nu}(z,k)}\ls\f{1}{|k|^2+|\Im z|^2}\quad\forall\Re z\geq-\bar{\lambda}|k|.\een
\end{prop}

\begin{proof}
	By Proposition \ref{S}(1), one has
	\[\nr{K^\nu(t,k)}\ls\left\{\ba
	&te^{-\delta'|k|^2t^2}\quad&\nu t\leq1\\
	&\f1\nu e^{-\delta'\f{|k|^2}{\nu}t}\quad&\nu t\geq1.
	\ea\right.\]
	Let $\bar{\lambda}<\delta'$. One can verify that
	\ben\label{ksquare}\nr{\wt{K^\nu}(z,k)}\ls\int_0^\infty\nr{K^\nu(t,k)}e^{\bar{\lambda}|k|t}dt\ls\f{1}{|k|^2},\quad\forall\Re z\geq-\bar{\lambda}|k|.\een
	Thus there exists sufficiently large $k_0$ and $\kappa>0$ such that
	\ben\label{k0}\inf_{|k|\geq k_0}\inf_{\Re z\geq-\bar{\lambda}|k|}\nr{\wt{K^\nu}(z,k)+1}\geq\kappa.\een
	For fixed $0<|k|\leq k_0$, it hold that
	\[\sup_{\Re z\geq-\bar{\lambda}|k|}\nr{\wt{K^\nu}(z,k)-\wt{K^0}(z,k)}\leq\int_0^\infty\nr{K^\nu(t,k)-K^0(t,k)}e^{\bar{\lambda}|k|t}dt,\]
	\[\nr{K^\nu(t,k)-K^0(t,k)}e^{\bar{\lambda}|k|t}\leq\rr{\nr{K^\nu(t,k)}+\nr{K^0(t,k)}}e^{\bar{\lambda}|k|t}\in L^1(\R_+).\]
	Thus by dominated convergence theorem, it holds that
	\[\lim_{\nu\rightarrow0}\sup_{\Re z\geq-\bar{\lambda}|k|}\nr{\wt{K^\nu}(z,k)-\wt{K^0}(z,k)}=0.\]
	For $\nu=0$, there exists $\bar{\lambda}>0, \kappa>0$ such that
	\[\inf_{k\in \Z^d\backslash\{0\}}\inf_{\Re z\geq-\bar{\lambda}|k|}\nr{\wt{K^0}(z,k)+1}\geq2\kappa.\]
	Therefore there exists $\nu_0>0, \bar{\lambda}>0$ such that for any $\nu<\nu_0$,
	\ben\label{gk0}\inf_{0<|k|\leq k_0}\inf_{\Re z\geq-\bar{\lambda}|k|}\nr{\wt{K^\nu}(z,k)+1}\geq\inf_{0<|k|\leq k_0}\inf_{\Re z\geq-\bar{\lambda}|k|}\nr{\wt{K^0}(z,k)+1}\notag\\
	-\sup_{0<|k|\leq k_0}\sup_{\Re z\geq-\bar{\lambda}|k|}\nr{\wt{K^\nu}(z,k)-\wt{K^0}(z,k)}\geq\kappa.\een
	 We get \eqref{Penrose} by \eqref{k0} and \eqref{gk0}.
	
	To prove \eqref{Im},  since $K^\nu(0,k)=0$, by integration by part, it holds that
	\[\wt{K^\nu}(z,k)=\left.\rr{-\f{e^{-zt}}{z^2}\f{d}{dt}K^\nu(t,k) }\right|_0^\infty+\f{1}{z^2}\int_0^\infty e^{-zt}\f{d^2}{dt^2}K^\nu(t,k)dt.\]
	One can verify that
	\[\nr{\f{d}{dt}K^\nu(t,k)}\ls e^{-\delta'|k|t}\quad\nr{\f{d^2}{dt^2}K^\nu(t,k)}\ls|k|e^{-\delta'|k|t}.\]
	Thus
	\ben\label{Imsquare}\nr{\wt{K^\nu}(z,k)}\ls\f{1}{|\Im z|^2},\quad\Re z\geq-\bar{\lambda}|k|.\een
	 We arrive at \eqref{Im} by combining \eqref{ksquare} and \eqref{Imsquare}.
\end{proof}

Then we begin to prove \eqref{resolvent}.
\begin{proof}
	We can see from \eqref{Penrose} that
	\[\wt{R^\nu}(z,k)=\f{\wt{K^\nu}(z,k)}{1+\wt{K^\nu}(z,k)}\]
	is well-defined on $\{\Re z\geq-\bar{\lambda}|k|\}$. By inverse Laplace transform (note that \eqref{Im} ensures that the right-hand side is integrable on $\{\Re z=-\bar{\lambda}|k|\}$) it holds that
	\[\ba\nr{R^\nu(t,k)}=\f{1}{2\pi}\nr{\int_\R\f{\wt{K^\nu}(-\bar{\lambda}|k|+i\omega,k)}{\wt{K^\nu}(-\bar{\lambda}|k|+i\omega,k)+1}e^{(-\bar{\lambda}|k|+i\omega)t}d\omega}\ls\f{e^{-\bar{\lambda}|k|t}}{\kappa}\int_\R\f{d\omega}{|k|^2+|\omega|^2}\ls\f{1}{|k|}e^{-\bar{\lambda}|k|t}.\ea\]
\end{proof}

\subsection{Improving the estimate of the density} 
In what follows, we shall use \eqref{newforumladens}  and \eqref{Hnu} to improve \eqref{Hrho}.
By \eqref{resolvent} and $|k\tent-k\tentau|\leq|k|(t-\tau)$, we observe that
\[\ba
&\int_0^{T^*}\sum_k\rr{|k|^\f12A^\nu_\f12(t,k,k\tent)e^{\delta\nu^\f13t}\<\tent\>^4\int_0^t H^\nu(\tau,k)R^\nu(t-\tau,k)d\tau}^2dt\\
\ls&\int_0^{T^*}\sum_k\left(\int_0^t \nr{|k|^\f12A^\nu_\f12(\tau,k,k\tentau) H^\nu(\tau,k)e^{\delta\nu^\f13\tau}\<\tentau\>^4}\right.\\
	&\left.\times\nr{A^\nu_\f92(\tau,0,k\tent-k\tentau)R^\nu(t-\tau,k)e^{\delta\nu^\f13(t-\tau)}}d\tau\right)^2dt\\
\ls&\int_0^{T^*}\sum_k\int_0^te^{-\delta'|k|(t-\tau)}d\tau\int_0^t\rr{|k|^\f12A^\nu_\f12(\tau,k,k\tentau) H^\nu(\tau,k)e^{\delta\nu^\f13\tau}\<\tentau\>^4}^2e^{-\delta'|k|(t-\tau)}d\tau dt\\
\ls&\sum_k\int_0^{T^*}\rr{|k|^\f12A^\nu_\f12(\tau,k,k\tentau) H^\nu(\tau,k)e^{\delta\nu^\f13\tau}\<\tentau\>^4}^2\int_\tau^{T^*}e^{-\delta'|k|(t-\tau)}dtd\tau\\
\ls&\sum_k\int_0^{T^*}\rr{|k|^\f12A^\nu_\f12(\tau,k,k\tentau) H^\nu(\tau,k)e^{\delta\nu^\f13\tau}\<\tentau\>^4}^2d\tau.
\ea\]
Thus by \eqref{newforumladens}, we may reduce the desired result \eqref{Hrho} to the estimate of  $\LLTXT{|\na_x|^\f12A^\nu_\f12H^\nu(t)e^{\delta\nu^\f13t}\<\tent\>^4}$. Thanks to \eqref{Hnu}, we only need to estimate $\LLTXT{|\na_x|^\f12A^\nu_\f12\cI(t)e^{\delta\nu^\f13t}\<\tent\>^4}, \LLTXT{|\na_x|^\f12A^\nu_\f12\cN_{\neq}(t)e^{\delta\nu^\f13t}\<\tent\>^4}$ and $\LLTXT{|\na_x|^\f12A^\nu_\f12\cN_0(t)e^{\delta\nu^\f13t}\<\tent\>^4}$.

\subsubsection{Estimate of $\cI$} By Proposition \ref{S}, we have 
\[\ba
&\LLTXT{|\na_x|^{\f12}A^\nu_{\f12}\cI(t)e^{\delta\nu^{\f13}t}\<\tent\>^4}^2\ls\sum_k\int_0^{T^*}|k|\nr{\wh{A^\nu_\f92h}(0,k,k\tent)}^2e^{-\nu t}dt\\
=&\sum_k\int_0^{T^*}\nr{\wh{A^\nu_\f92h}(0,k,t')}^2dt'\ls\LLT{\<v\>^mA^\nu_\f92h(0)}^2\ls\eps^2.
\ea\]

\subsubsection{Estimate of $\cN_{\neq}$} 
First we address the time growth $\<\tent\>^4$. We split the time integral into two parts:
\[\ba
&|k|^\f12\wh{A^\nu_\f12\cN_{\neq}}(t,k)e^{\delta\nu^\f13t}\<\tent\>^4=-\sum_{\ell\neq k}\int_0^t\wh{\rho}(\tau,\ell)\f{\ell\cdot k}{|\ell|^2}\f{1-e^{-\nu(t-\tau)}}{\nu}\wh{f}(\tau,k-\ell,k\tent-\ell\tentau)\\
&S(t-\tau,k)|k|^\f12A^\nu_\f12(\tau,k,k\tent)e^{\delta\nu^\f13t}\<\tent\>^4e^{\lambda(\tent,|k,k\tent|)-\lambda(\tentau,|k,k\tentau|)}\\
&\rr{\vv{1}_{\tentau\leq\f12\tent}+\vv{1}_{\tentau\geq\f12\tent}}d\tau:=\mathcal{B}_1+\mathcal{B}_2.
\ea\]
For $\mathcal{B}_2$ we notice that $\tent\sim\tentau$. For $\mathcal{B}_1$ since $\tent-\tentau\geq\f12\tent$, it holds that
\[\ba&\lambda(\tentau,|k,k\tent|)-\lambda(\tent,|k,k\tent|)\gs \rr{(1+\tentau)^{-b}-(1+\tent)^{-b}}\\\times&\<k,k\tent\>^s
\gs \<\tent\>^{-b}\<k,k\tent\>^s\gs \<\tent\>^{s-b}.\ea\]
The time growth can be absorbed by $e^{-\delta'\<\tent\>^{s-b}}$. Thus the following estimate holds
\[\ba
|k|^\f12\nr{\wh{A^\nu_\f12\cN_{\neq}}(t,k)}e^{\delta\nu^\f13t}\<\tent\>^4\ls\sum_{\ell\neq k}\int_0^t\nr{\wh{\rho}(\tau,\ell)}\f{|k|}{|\ell|}\f{1-e^{-\nu(t-\tau)}}{\nu}\nr{\wh{f}(\tau,k-\ell,k\tent-\ell\tentau)}\\
S(t-\tau,k)|k|^\f12A^\nu_\f12(\tau,k,k\tent)e^{\delta\nu^\f13t}\<\tentau\>^4e^{\f12\lambda(\tent,|k,k\tent|)-\f12\lambda(\tentau,|k,k\tentau|)}d\tau.
\ea\]

Then we use the standard paraproduct decomposition to split $\cN_{\neq}$ into its HL part and LH part.
\[\ba
&\wh{\cN_{\neq}}(t,k)=-\sum_{\ell\neq k}\int_0^t\wh{\rho}(\tau,\ell)\f{\ell\cdot k}{|\ell|^2}\f{1-e^{-\nu(t-\tau)}}{\nu}\wh{f}(\tau,k-\ell,k\tent-\ell\tentau)S(t-\tau,k)\\
&\rr{\vv{1}_{\<\ell,\ell\tentau\>\leq\<k-\ell,k\tent-\ell\tentau\>}+\vv{1}_{\<\ell,\ell\tentau\>\geq\<k-\ell,k\tent-\ell\tentau\>}}d\tau:=\wh{\cN_{\neq,LH}}(t,k)+\wh{\cN_{\neq,HL}}(t,k).
\ea\]
$\bullet$\underline{\emph{Estimate of the LH part}.} We note that $\<k,k\tent\>\leq2\<k-\ell,k\tent-\ell\tentau\>$ on $\{\<\ell,\ell\tentau\>\leq\<k-\ell,k\tent-\ell\tentau\>\}$.
By the bootstrap hypotheses \eqref{Topf} and \eqref{Hed} and $|\bar{\eta}(t,k,\eta)|\leq \ent\<k,\eta\>\<t\>$, we first have
\[\LLT{\<v\>^m\vect A^\nu_{-1}f_{\neq}}\ls\eps e^{-\delta_1\nu^{\f13}t} e^{(m+2)\nu t}\<t\>,\quad \LLT{\<v\>^m\vect A^\nu_{1}f}\ls\eps e^{(m+2)\nu t}\<t\>.\]
Thus by interpolation, one has
\ben\label{interpineq1} &&\LLT{\<v\>^m\vect A^\nu_{\f12}f_{\neq}}^2
=\sum_{k\neq0}\nnr{\bar{\eta}\wh{A^\nu_{\f12}f}(t,k,\eta)}_{H^m_\eta}^2\ls\LLT{\<v\>^m\vect A^\nu_{-1}f_{\neq}}^{\f12}\LLT{\<v\>^m\vect A^\nu_{1}f}^{\f32}\notag\\&&\ls \eps^2\<t\>^{2}e^{-\delta'\nu^{\f13}t}.\een
Using Proposition \ref{S}  and Proposition \ref{nut}, we infer that
\[\ba
&\rr{\wh{|k|^{\f12}A^\nu_{\f12}\cN_{\neq, LH}(t,k)}e^{\delta\nu^{\f13}t}\<\tent\>^4}^2\ls\left(\sum_{\ell\neq k}\int_0^t\nr{\wh{A^\nu_{\f12}\rho}(\tau,\ell)}e^{\delta\nu^{\f13}\tau}\<\tentau\>^4\<\ell,\ell\tentau\>^{-\beta-\f12}\right.\\
&\left.\times|k|^{\f12}\nr{\wh{\vect A^\nu_{\f12}f}(\tau,k-\ell,k\tent-\ell\tentau)}S^{\f12}(t-\tau,k)d\tau\right)^2\\
&\ls \big(\sum_\ell\int_0^t\nr{\wh{A^\nu_{\f12}\rho}(\tau,\ell)}|\ell|^\f12e^{\delta\nu^{\f13}\tau}\<\tentau\>^4\<\ell\>^{-\f{d}{2}-\f12}\<\tentau\>^{-1}\<\nu\tau\>^{-1}d\tau\big)\ \sum_{\ell\neq k}\int_0^t\nr{\wh{A^\nu_{\f12}\rho}(\tau,\ell)}|\ell|^\f12 e^{\delta\nu^{\f13}\tau}\\&\times\<\tentau\>^4\<\ell\>^{-\f{d}{2}-\f12}\<\tentau\>^{-3}\<\nu\tau\>^{-3}
|k|\nr{\wh{\vect A^\nu_{\f12}f}(\tau,k-\ell,k\tent-\ell\tentau)}^2S(t-\tau,k)\<\nu\tau\>^4 d\tau\\
&\ls\nnr{|\na_x|^\f12A^\nu_{\f12}\rho(t)e^{\delta\nu^{\f13}t}\<\tent\>^4}_{L^2([0,T^*],L^2_x)}\sum_{\ell\neq k}\int_0^t\nr{\wh{A^\nu_{\f12}\rho}(\tau,\ell)}|\ell|^\f12 e^{\delta\nu^{\f13}\tau}\<\tentau\>^4\<\ell\>^{-\f{d}{2}-\f12}\<\tau\>^{-3}\\
&\times|k|\nr{\wh{\vect A^\nu_{\f12}f}(\tau,k-\ell,k\tent-\ell\tentau)}^2S(t-\tau,k)\<\nu\tau\>^{4} d\tau.
\ea\]
By \eqref{interpineq1}, we get that 
\[\ba
&\nnr{|\na_x|^{\f12}A^\nu_{\f12}\cN_{\neq, LH}(t)e^{\delta\nu^{\f13}t}\<\tent\>^4}_{L^2([0,T^*],L^2_x)}^2\ls \eps\sum_k\int_0^{T*}\sum_{\ell\neq k}\int_0^t\nr{\wh{A^\nu_{\f12}\rho}(\tau,\ell)}|\ell|^\f12e^{\delta\nu^{\f13}\tau}\<\tentau\>^4\\
&\times \<\ell\>^{-\f{d}{2}-\f12}\<\tau\>^{-3}e^{2\nu\tau}
|k|\nr{\wh{\vect A^\nu_{\f12}f}(\tau,k-\ell,k\tent-\ell\tentau)}^2e^{-\nu t} d\tau
dt\\
=&\eps\sum_\ell\int_0^{T^*}\nr{\wh{A^\nu_{\f12}\rho}(\tau,\ell)}|\ell|^\f12e^{\delta\nu^{\f13}\tau}\<\tentau\>^4\<\ell\>^{-\f{d}{2}-\f12}\<\tau\>^{-1}\sum_{k\neq\ell}\int_\tau^{T^*}\<\tau\>^{-2}e^{2\nu\tau}\\
&\times|k|\nr{\wh{\vect A^\nu_{\f12}f}(\tau,k-\ell,k\tent-\ell\tentau)}^2e^{-\nu t} dt
d\tau\\
\leq&\eps\sum_\ell\int_0^{T^*}\nr{\wh{A^\nu_{\f12}\rho}(\tau,\ell)}|\ell|^\f12e^{\delta\nu^{\f13}\tau}\<\tentau\>^4\<\ell\>^{-\f{d}{2}-\f12}\<\tau\>^{-1}\sum_{k\neq\ell}\int_\R\<\tau\>^{-2}e^{2\nu\tau}\nr{\wh{\vect A^\nu_{\f12}f}(\tau,k-\ell,\f{k}{|k|}t'-\ell\tentau)}^2 dt'
d\tau\\
\ls&\eps\sum_\ell\int_0^{T^*}\nr{\wh{A^\nu_{\f12}\rho}(\tau,\ell)}|\ell|^\f12e^{\delta\nu^{\f13}\tau}\<\tentau\>^4\<\ell\>^{-\f{d}{2}-\f12}\<\tau\>^{-1}\rr{\<\tau\>^{-2}e^{2\nu\tau}\sum_{k\neq\ell}\nnr{\bar{\eta}\wh{A^\nu_{\f12}f}(\tau,k,\eta)}_{H^m_\eta}^2}d\tau\\
\ls& \eps^3\nnr{|\na_x|^\f12A^\nu_{\f12}\rho(t)e^{\delta\nu^{\f13}t}\<\tent\>^4}_{L^2([0,T^*],L^2_x)}\ls \eps^4.
\ea\]

$\bullet$\underline{\emph{Estimate of the HL part}.} We note that $\<k,k\tent\>\leq2\<\ell,\ell\tentau\>$ on $\{\<\ell,\ell\tentau\>\geq\<k-\ell,k\tent-\ell\tentau\>\}$.
Observing that $\<k\>^{\f12}\ls \<\ell\>^{\f12}\<k-\ell\>^{\f12}$ and $|k\f{1-e^{-\nu(t-\tau)} }{\nu}|\leq \<(k-\ell, k\tent-\ell\tentau)\>\entau\<\tentau\>$, we have
\[\ba
&\rr{\wh{|k|^{\f12}A^\nu_{\f12}\cN_{\neq, HL}(t,k)}e^{\delta\nu^{\f13}t}\<\tent\>^4}^2\ls \left(\sum_{\ell\neq k}\int_0^t\nr{\wh{A^\nu_{\f12}\rho}(\tau,\ell)}|\ell|^{\f12}e^{\delta\nu^{\f13}\tau}\<\tentau\>^4\f{|k|}{|\ell|}\f{1-e^{-\nu(t-\tau)} }{\nu}\entau\right.\\
&\times\nr{\wh{A^\nu f}(\tau,k-\ell,k\tent-\ell\tentau)}\<k-\ell, k\tent-\ell\tentau\>^{-\beta+\f12}S^{\f34}(t-\tau,k)\\
&\left.\times e^{\f12\lambda(\tent,|k,k\tent|)-\f12\lambda(\tentau,|k,k\tent|)}\emnt d\tau\right)^2\\
\ls&\left(\sum_{\ell\neq k}\int_0^t\nr{\wh{A^\nu_{\f12}\rho}(\tau,\ell)}|\ell|^{\f12}e^{\delta\nu^{\f13}\tau}\<\tentau\>^4\f{1}{|\ell|}\<\tentau\>e^{2\nu\tau}e^{\f{\delta_1} {2}\nu^\f13\tau}\nr{\wh{A^\nu f}(\tau,k-\ell,k\tent-\ell\tentau)}\right.\\
&\left.\times \<k-\ell, k\tent-\ell\tentau\>^{-\beta+\f32}S^{\f12}(t-\tau,k)e^{\f12\lambda(\tent,|k,k\tent|)-\f12\lambda(\tentau,|k,k\tent|)}\emnt e^{-\f{\delta_1}{2} \nu^\f13t}d\tau\right)^2.
\ea\]
Let
\ben\label{Kkell}  K^\nu_{k,\ell}(t,\tau):=\f{1}{|\ell|}\<\tentau\>\<k-\ell, k\tent-\ell\tentau\>^{-\beta+\f32}\notag\\\times S^{\f12}(t-\tau,k)e^{\f12\lambda(\tent,|k,k\tent|)-\f12\lambda(\tentau,|k,k\tent|)}e^{-\f{\delta_1}{2} \nu^\f13t}\emnt.
\een
Then it holds that
\[\ba
&\nnr{|\na_x|^{\f12}A^\nu_{\f12}\cN_{\neq, HL}(t)e^{\delta\nu^{\f13}t}\<\tent\>^4}_{L^2([0,T^*],L^2_x)}^2\ls\sum_k\int_0^{T^*}\rr{\sum_{\ell\neq k}\int_0^t\nr{\wh{A^\nu_{\f12}\rho}(\tau,\ell)}|\ell|^{\f12}e^{\delta\nu^{\f13}\tau}\<\tentau\>^4K^\nu_{k,\ell}(t,\tau)d\tau}^2dt\\
&\times\rr{\sup_\tau\sup_{k\neq0, \eta}\nr{\wh{A^\nu f}(\tau,k,\eta)}e^{2\nu\tau}e^{\f{\delta_1} {2}\nu^\f13\tau}}^2\ls\eps^2\nnr{|\na_x|^{\f12}A^\nu_{\f12}\rho(t)e^{\delta\nu^{\f13}t}\<\tent\>^4}_{L^2([0,T^*],L^2_x)}^2\\
&\times\rr{\sup_{k\neq0, t}\int_0^t\sum_{\ell\neq k,0}K^\nu_{k,\ell}(t,\tau)d\tau}\rr{\sup_{\ell\neq0, \tau}\int_\tau^{T^*}\sum_{k\neq\ell,0}K^\nu_{k,\ell}(t,\tau)dt}.
\ea\]
We claim that
\[\nnr{|\na_x|^{\f12}A^\nu_{\f12}\cN_{\neq, HL}(t)e^{\delta\nu^{\f13}t}\<\tent\>^4}_{L^2([0,T^*],L^2_x)}^2\ls\left\{\ba
&\eps^4\quad &s\geq\f13\\&\eps^4\nu^{-2\f{1-3s}{3-3s}}\quad&0<s<\f13.\ea\right.\]
To prove the claim, we only need to prove the following proposition:
\begin{prop} It holds  that
	\[\sup_{k\neq0, t\le T^*}\int_0^t\sum_{\ell\neq k,0}K^\nu_{k,\ell}(t,\tau)d\tau+ \ \sup_{\ell\neq0, \tau\le T^*}\int_\tau^{T^*}\sum_{k\neq\ell,0}K^\nu_{k,\ell}(t,\tau)dt\ls\left\{\ba
		&1\quad &s\geq\f13\\&\nu^{-\f{1-3s}{3-3s}}\quad&0<s<\f13.\ea\right.
	\]
\end{prop}
\begin{proof}
	
	We provide a detailed estimate for the first term, and the second term can be estimated using a similar approach. We focus on the case $t\geq1$. We divide this term into three parts.
    \[\ba
    &\int_0^t\sum_{\ell\neq k,0}\f{1}{|\ell|}\<\tentau\>\<k-\ell, k\tent-\ell\tentau\>^{-\beta+\f32}e^{\f12\lambda(\tent,|k,k\tent|)-\f12\lambda(\tentau,|k,k\tent|)} e^{-\f{\delta_1}{2} \nu^\f13t}\\&\times\emnt
    \rr{\vv{1}_{|k\tent-\ell\tentau|\geq\f1{100}\tent}+\vv{1}_{\substack{|k\tent-\ell\tentau|\leq\f1{100}\tent\\\text{and}\ |\ell-k|>\f1{10}|k|}}+\vv{1}_{\substack{|k\tent-\ell\tentau|\leq\f1{100}\tent\\\text{and}\ |\ell-k|\leq\f1{10}|k|}} }d\tau\\&
    :=\mathcal{A}_{nr}^1+\mathcal{A}_{nr}^2+\mathcal{A}_{r}.    
    \ea\]
    We first consider $\mathcal{A}_{nr}^1$.
    \[\ba\mathcal{A}_{nr}^1\ls\sum_{\ell\neq k,0}\int_0^t\f{1}{|\ell|}\<\tentau\>\<\ell-k\>^{-d-\f12}\<k\tent-\ell\tentau\>^{-\f32}\<\tent\>^{-1}\emntau d\tau\ls\sum_{\ell\neq k,0}\f{1}{|\ell|}\<\ell-k\>^{-d-\f12}\ls1.\ea\]
    For $\mathcal{A}_{nr}^2$ we notice that $\tent-\tentau\geq\f{|\ell-k|-0.01}{|\ell|}\tent$ on ${|k\tent-\ell\tentau|\leq\f1{100}\tent}$ and $|\ell-k|\gs|\ell|$ on $\{|\ell-k|\geq0.1|k|\}$, which implies that
    \[\ba&\lambda(\tentau,|k,k\tent|)-\lambda(\tent,|k,k\tent|)\gs \rr{(1+\tentau)^{-b}-(1+\tent)^{-b}}\\\times&\<k,k\tent\>^s
    \gs \f{|\ell-k|}{|\ell|\<\tent\>^b}\<k,k\tent\>^s\gs \<\tent\>^{s-b}.\ea\]
    Thus it holds that
    \[\mathcal{A}_{nr}^2\ls\sum_{\ell\neq k,0}\int_0^t\f{1}{|\ell|}\<\tentau\>\<\ell-k\>^{-d-\f12}\<k\tent-\ell\tentau\>^{-\f32}e^{-\delta'\<\tent\>^{s-b}}\emntau d\tau\ls1.\]
    For $\mathcal{A}_r$ we notice that $\tent\ls\f{|k|-\f12}{|\ell|}\tent\ls\tentau\leq\f{|k|+\f12}{|\ell|}\tent\ls\tent$ on $\{|k\tent-\ell\tentau|\leq\f1{100}\tent\ \text{and}\ |\ell-k|\leq\f1{10}|k|\}$. We consider separately the case $s\geq\f13$ and $0<s<\f13$.
    
    For $s\geq\f13$, we consider the following two cases:
    
    $\bullet$ $\tent\leq|k|^2$. In this case we have
    \[\mathcal{A}_r\ls\sum_{\ell\neq k,0}\f{1}{|k|}\<\tent\>\<\ell-k\>^{-d-\f12}\int_0^t\<k\tent-\ell\tentau\>^{-\f32}\emntau d\tau\ls\sum_{\ell\neq k,0}\f{1}{|k|^2}\<\tent\>\<\ell-k\>^{-d-\f12}\ls1.\]
    
    $\bullet$ $\tent\geq|k|^2$.  We notice that in this case $\tent-\tentau\gs\f{|\ell-k|}{|\ell|}\tent$ and $|\ell|\sim|k|$. Thus we have
    \ben&&\lambda(\tentau,|k,k\tent|)-\lambda(\tent,|k,k\tent|)\gs \left((1+\tentau\<k,k\tent\>^{s-1})^{-b}\right.\notag\\&&
    \left.-(1+\tent\<k,k\tent\>^{s-1})^{-b}\right)\<k,k\tent\>^s\gs\f{|\ell-k|}{|\ell|\<\tent\>^b}\<k,k\tent\>^{s+b(1-s)}\notag\\&&\gs|\ell-k|\rr{\tent|k|^{1-\f1s}}^{s(1-b)}.\label{lambdatt}\een
    Thus it holds that
    \[\ba
    \mathcal{A}_r&\ls\sum_{\substack{\ell\neq k,0\\|\ell-k|\leq0.1|k|}}\f{1}{|k|}\<\tent\>\<k-\ell\>^{-d-\f12}\int_0^t\<k\tent-\ell\tentau\>^{-\f32}e^{-\delta'\rr{\tent|k|^{1-\f1s}}^{s(1-b)}}\emntau d\tau\\&\ls\f{1}{|k|^2}\<\tent\>e^{-\delta'\rr{\tent|k|^{1-\f1s}}^{s(1-b)}}\ls1.
    \ea\]
    
    For $0<s<\f13$, we consider the following three cases:
    
    $\bullet$ $|k|\geq \nu^{-\f{s}{3-3s}}$. In this case $\mathcal{A}_r\ls\f{t}{|k|^2}e^{-\f{\delta_1}2\nu^{\f13}t}\ls\nu^{-\f13}|k|^{-2}\ls\nu^{-\f{1-3s}{3-3s}}$.
    
    $\bullet$ $|k|\leq \nu^{-\f{s}{3-3s}}$ and $\tent\leq|k|^{\f1s-1}$. In this case $\mathcal{A}_r\ls\f1{|k|^2}\tent\ls|k|^{\f1s-3}\ls\nu^{-\f{1-3s}{3-3s}}$.
    
    $\bullet$ $|k|\leq \nu^{-\f{s}{3-3s}}$ and $\tent\geq|k|^{\f1s-1}$. We notice that in this case \eqref{lambdatt} holds. Thus we have
    \[\ba
    \mathcal{A}_r&\ls\sum_{\substack{\ell\neq k,0\\|\ell-k|\leq0.1|k|}}\f{1}{|k|}\<\tent\>\<k-\ell\>^{-d-\f12}\int_0^t\<k\tent-\ell\tentau\>^{-\f32}e^{-\delta'\rr{\tent|k|^{1-\f1s}}^{s(1-b)}}\emntau d\tau\\&\ls\f{1}{|k|^2}\<\tent\>e^{-\delta'\rr{\tent|k|^{1-\f1s}}^{s(1-b)}}\ls|k|^{\f1s-3}\ls\nu^{-\f{1-3s}{3-3s}}.
    \ea\]
\end{proof}

  We may conclude that 
  \beno \LLTXT{|\na_x|^\f12A^\nu_\f12\cN_{\neq}(t)e^{\delta\nu^\f13t}\<\tentau\>^4}^2\ls\left\{\ba
  &\eps^4\quad &s\geq\f13\\&\eps^4\nu^{-2\f{1-3s}{3-3s}}\quad&0<s<\f13.\ea\right. \eeno 

\subsubsection{Estimate of $\cN_{0}$}
We first use the standard paraproduct decomposition to split it into its LH part and HL part.
\[\ba
&\wh{\cN_{0}}(t,k)=-\int_0^t\wh{\rho}(\tau,k)\f{1-e^{-\nu(t-\tau)}}{\nu}\wh{h}(\tau,0,k\f{1-e^{-\nu(t-\tau)}}{\nu})S(t-\tau,k)\left(\vv{1}_{3\<k,k\tentau\>\leq\<k\f{\emntau-\emnt}{\nu}\> }\right.\\
&\left.+\vv{1}_{3\<k,k\tentau\>\geq\<k\f{\emntau-\emnt}{\nu}\>}\right)d\tau
:=\wh{\cN_{0,LH}}(t,k)+\wh{\cN_{0,HL}}(t,k).
\ea\]

$\bullet$ \underline{\emph{Estimate of the LH part}.}
Let $\vf$ be a radial smooth function satisfying $\vv{1}_{B(0,2)^c}\leq\vf\leq\vv{1}_{B(0,1)^c}$. By the bootstrap hypotheses \ref{Topf} and \ref{Hsh}, we have
\beno 
&&\qquad \left\|A^\nu_2(t,\emnt\eta)\wh{h_0}(t,\eta)\right\|_{H^m_\eta}\ls\eps e^{(1+\f d2)\nu t}, \\
  &&\mbox{and}\quad \nnr{A^\nu_1(t,\emnt\eta)\wh{h_0}(t,\eta)\vf(\emnt\eta)}_{H^m_\eta}\ls\nnr{\<\eta\>^{\beta+1}A^\nu_{-\beta}(t,\emnt\eta)\wh{h_0}(t,\eta)}_{H^m_\eta}e^{-(\beta+1)\nu t} \\&&\qquad\ls\nnr{A^\nu_{-\beta}(t,\emnt\na)h_0(t)}_{H^{\beta+1}_{m}}e^{-(\beta+1)\nu t}\ls\eps e^{-(\beta+1)\nu t}.
\eeno 
By interpolation and $\beta\geq\f d2+m+3$, we get that
\[\nnr{A^\nu_\f32(t,\emnt\eta)\wh{h_0}(t,\eta)\vf(\emnt\eta)}^2_{H^m_\eta}\ls\eps^2e^{-3\nu t}.\]
Then we begin to estimate $\cN_{0, LH}$.
\[\ba
&\rr{\wh{|\na_x|^{\f12}A^\nu_{\f12}\cN_{0, LH}(t,k)}e^{\delta\nu^{\f13}t}\<\tent\>^4}^2\ls\left(\int_0^t\nr{\wh{A^\nu_{\f12}\rho}(\tau,k)}e^{\delta\nu^{\f13}\tau}\<\tentau\>^4\<k\>^{-\f{d}{2}-\f12}\<\tentau\>^{-1}\entau\right.\\
&\left.\times|k|^{\f12}\nr{A^\nu_{\f32}(\tau,k\f{\emntau-\emnt}{\nu})\wh{h_0}(\tau,k\f{1-e^{-\nu(t-\tau)}}{\nu})}S^{\f12}(t-\tau,k)\vf(k\f{\emntau-\emnt}{\nu}) d\tau\right)^2\\
\ls&\int_0^t\nr{\wh{A^\nu_{\f12}\rho}(\tau,k)}e^{\delta\nu^{\f13}\tau}\<\tentau\>^4\<k\>^{-\f{d}{2}-\f12}\<\tau\>^{-1}d\tau\int_0^t\nr{\wh{A^\nu_{\f12}\rho}(\tau,k)}e^{\delta\nu^{\f13}\tau}\<\tentau\>^4\<k\>^{-\f{d}{2}-\f12}\<\tau\>^{-1}e^{3\nu \tau}\\
&\times|k|\nr{A^\nu_{\f32}(\tau,k\f{\emntau-\emnt}{\nu})\wh{h_0}(\tau,k\f{1-e^{-\nu(t-\tau)}}{\nu})}^2S(t-\tau,k)\vf^2(k\f{\emntau-\emnt}{\nu})d\tau\\
\ls&\nnr{A^\nu_{\f12}\rho(t)e^{\delta\nu^{\f13}t}\<\tent\>^4}_{L^2([0,T^*],L^2_x)}\int_0^t\nr{\wh{A^\nu_{\f12}\rho}(\tau,k)}e^{\delta\nu^{\f13}\tau}\<\tentau\>^4\<k\>^{-\f{d}{2}-\f12}\<\tau\>^{-1}e^{3\nu \tau}\\
&\times|k|\nr{A^\nu_{\f32}(\tau,k\f{\emntau-\emnt}{\nu})\wh{h_0}(\tau,k\f{1-e^{-\nu(t-\tau)}}{\nu})}^2S(t-\tau,k)\vf^2(k\f{\emntau-\emnt}{\nu})d\tau.\\
\ea\]
From this, we have 
\[\ba
&\nnr{|\na_x|^{\f12}A^\nu_{\f12}\cN_{0, LH}(t)e^{\delta\nu^{\f13}t}\<\tent\>^4}_{L^2([0,T^*],L^2_x)}^2\ls\eps\sum_k\int_0^{T^*}\nr{\wh{A^\nu_{\f12}\rho}(\tau,k)}e^{\delta\nu^{\f13}\tau}\<\tentau\>^4\<k\>^{-\f{d}{2}-\f12}\<\tau\>^{-1}\\
&\times\int_\tau^{T^*}e^{3\nu \tau}|k|\nr{A^\nu_{\f32}(\tau,k\f{\emntau-\emnt}{\nu})\wh{h_0}(\tau,k\f{1-e^{-\nu(t-\tau)}}{\nu})}^2\vf^2(k\f{\emntau-\emnt}{\nu})e^{-\nu(t-\tau)}dtd\tau\\
\ls&\eps\sum_k\int_0^{T^*}\nr{\wh{A^\nu_{\f12}\rho}(\tau,k)}e^{\delta\nu^{\f13}\tau}\<\tentau\>^4\<k\>^{-\f{d}{2}-\f12}\<\tau\>^{-1}d\tau\rr{\sup_\tau e^{3\nu \tau}\nnr{A^\nu_\f32(\tau,\emntau\eta)\wh{h_0}(\tau,\eta)\vf(\emnt\eta)}^2_{H^m_\eta}}\\
\ls&\eps^4.
\ea\]

$\bullet$ \underline{\emph{Estimate of the HL part}.} By direct calculation, one has
\[\ba
&\nnr{|\na_x|^{\f12}A^\nu_{\f12}\cN_{0, HL}(t)e^{\delta\nu^{\f13}t}\<\tent\>^4}_{L^2([0,T^*],L^2_x)}^2\ls\int_0^{T^*}\sum_k\left(\int_0^t\nr{\wh{A^\nu_\f12\rho}(\tau,k)}|k|^\f12e^{\delta\nu^{\f13}\tau}\<\tentau\>^4\right.\\
&\left.\times\f{1-e^{-\nu(t-\tau)}}{\nu}\nr{A^\nu_{-\beta}(\tau,k\f{\emntau-\emnt}{\nu})\wh{h_0}(\tau,k\f{1-e^{-\nu(t-\tau)}}{\nu})}S^\f12(t-\tau,k)d\tau\right)^2dt\\
\ls&\int_0^{T^*}\sum_k\rr{\int_0^t\nr{\wh{A^\nu_\f12\rho}(\tau,k)}|k|^\f12e^{\delta\nu^{\f13}\tau}\<\tentau\>^4\<\f{1-e^{-\nu(t-\tau)}}{\nu}\>^{-2}S^\f12(t-\tau,k)d\tau}^2dt\\
&\times\rr{\sup_\tau \nnr{A^\nu_{-\beta}(\tau,\emntau\eta)\<\eta\>^3\wh{h_0}(\tau,\eta)}^2_{H^m_\eta}}\\
\ls&\int_0^{T^*}\sum_k\int_0^t\nr{\wh{A^\nu_\f12\rho}(\tau,k)}^2|k|e^{2\delta\nu^{\f13}\tau}\<\tentau\>^4\<\f{1-e^{-\nu(t-\tau)}}{\nu}\>^{-2}e^{-\nu(t-\tau)}d\tau\int_0^t\<\f{1-e^{-\nu(t-\tau)}}{\nu}\>^{-2}e^{-\nu(t-\tau)}d\tau dt\\
&\times\rr{\sup_\tau \nnr{A^\nu_{-\beta}(\tau,\emntau\na)h_0(\tau)}_{H^3_m}^2}\\
\ls&\eps^2\int_0^{T^*}\sum_k\nr{\wh{A^\nu_\f12\rho}(\tau,k)}^2|k|e^{2\delta\nu^{\f13}\tau}\<\tentau\>^4\int_\tau^{T^*}\<\f{1-e^{-\nu(t-\tau)}}{\nu}\>^{-2}e^{-\nu(t-\tau)}dtd\tau\\
\ls&\eps^2\nnr{|\na|^\f12A^\nu_{\f12}\rho(t)e^{\delta\nu^{\f13}t}\<\tent\>^4}^2_{L^2([0,T^*],L^2_x)}\ls\eps^4.
\ea\] 

We conclude that
 \beno \LLTXT{|\na_x|^\f12A^\nu_\f12\cN_{0}(t)e^{\delta\nu^\f13t}\<\tent\>^4}^2\ls \varepsilon^4. \eeno 

\medskip
By putting together all the estimates, we get that
\ben\label{improveHrho}
&&\LLTXT{|\na_x|^\f12A^\nu_\f12\rho(t)e^{\delta\nu^\f13t}\<\tent\>^4}\leq C_1(m,d)\eps\notag\\&&+C_2(K_{\rho }, K_f, K_{ED}, K_{SH})\left\{\ba&
\eps^2\quad &s\geq\f13\\&\eps^2\nu^{-\f{1-3s}{3-3s}}\quad&0<s<\f13.\ea\right.
\een

\section{Top order distribution function estimate}

In this section we will improve \eqref{Topf}. We   choose
\[E^T(t):=\sum_{|\alpha|\leq m}\f{e^{-2(|\alpha|+1)\nu t}}{K_{|\alpha|}}\LLT{A^\nu_2(v^\alpha f)(t)}^2 \]as the energy functional. Since
\[\ba
&\pa_tD^\alpha_\eta\wh{f}(t,k,\eta)+\wh{\rho}(t,k)\f{k\cdot}{|k|^2}D^\alpha_\eta\rr{\eta\wh{\mu}(\eta)}(\bar{\eta})e^{|\alpha|\nu t}+\sum_\ell\wh{\rho}(t,\ell)\f{\ell\cdot\bar{\eta}}{|\ell|^2}D^\alpha_\eta\wh{f}(t,k-\ell,\eta-\ell\tent)\\
&+\sum_{\substack{\alpha'<\alpha\\|\alpha'|=|\alpha|-1}}C_{\alpha}^{\alpha'}\sum_\ell\wh{\rho}(t,\ell)\f{\ell_{\alpha-\alpha'}}{|\ell|^2}D^{\alpha'}_\eta\wh{f}(t,k-\ell,\eta-\ell\tent)\ent\\
&=-\nu\bar{\eta}^2D^\alpha_\eta\wh{f}(t,k,\eta)-\nu\sum_{\substack{\alpha'<\alpha\\|\alpha'|\geq|\alpha|-2}}C_{\alpha}^{\alpha'}D^{\alpha-\alpha'}_\eta\rr{\bar{\eta}^2}D^{\alpha'}_\eta\wh{f}(t,k,\eta),
\ea\]
basic energy method implies that
\[\f12\f{d}{dt}\rr{\LLT{A^\nu_2(v^\alpha f)(t)}^2e^{-2(|\alpha|+1)\nu t}}+\sD^{T,\alpha}+\sD^{T,\alpha}_{|\alpha|}+\sD^{T,\alpha}_\lambda=\sL^{T,\alpha}+\sN^{T,\alpha}_{top}+\sN^{T,\alpha}_{re}+\sR^{T,\alpha}_{re}.\] The terms are defined as follows: 
 $\sD^{T,\alpha}:=\nu\LLT{\vect A^\nu_2(v^\alpha f)(t)}^2e^{-2(|\alpha|+1)\nu t} $, $\sD^{T,\alpha}_{|\alpha|}:=\nu\rr{|\alpha|+1}\LLT{A^\nu_2(v^\alpha f)(t)}^2e^{-2(|\alpha|+1)\nu t} $, 
 \[\ba
&\sD^{T,\alpha}_\lambda:=\f{\lambda_1-\lambda_{\infty}}{8}b(1+\tent)^{-b-1}\emnt\LLT{A^\nu_{2+\f s2}(v^\alpha f)(t)}^2e^{-2(|\alpha|+1)\nu t}\\
&+\f{\lambda_1-\lambda_{\infty}}{8}b\emnt\LLT{(1+\tent\<\na\>^{s-1})^{-\f{b}{2}-\f12}A^\nu_{2+s-\f12}(v^\alpha f)(t)}^2e^{-2(|\alpha|+1)\nu t},\\
&\sL^{T,\alpha}:=-\sum_k \int\wh{\rho}(t,k)\f{k\cdot}{|k|^2}D^\alpha_\eta\rr{\eta\wh{\mu}(\eta)}(\bar{\eta})A^\nu_2(t,k,\eta)A^\nu_2\ov{D^\alpha_\eta\wh{f}(t,k,\eta)}d\eta\ \emnt e^{-(|\alpha|+1)\nu t} ,\\
&\sN^{T,\alpha}_{top}:=-\sum_{k,\ell}\int\wh{\rho}(t,\ell)\f{\ell\cdot\bar{\eta}}{|\ell|^2}D^\alpha_\eta\wh{f}(t,k-\ell,\eta-\ell\tent)A^\nu_2(t,k,\eta)A^\nu_2\ov{D^\alpha_\eta\wh{f}(t,k,\eta)}d\eta\ e^{-2(|\alpha|+1)\nu t} ,\\
&\sN^{T,\alpha}_{re}:=-\sum_{\substack{\alpha'<\alpha\\|\alpha'|=|\alpha|-1}}C_{\alpha}^{\alpha'}\sum_{k,\ell}\int\wh{\rho}(t,\ell)\f{\ell_{\alpha-\alpha'}}{|\ell|^2}D^{\alpha'}_\eta\wh{f}(t,k-\ell,\eta-\ell\tent)A^\nu_2(t,k,\eta)A^\nu_2\ov{D^\alpha_\eta\wh{f}(t,k,\eta)}d\eta\ e^{-(2|\alpha|+1)\nu t} ,\\
&\sR^{T,\alpha}_{re}:=\nu\sum_{\substack{\alpha'<\alpha\\|\alpha'|\geq|\alpha|-2}}C_{\alpha}^{\alpha'}\sum_k\int D^{\alpha-\alpha'}_\eta\rr{\bar{\eta}^2}A^\nu_2D^{\alpha'}_\eta\wh{f}(t,k,\eta)A^\nu_2\ov{D^\alpha_\eta\wh{f}(t,k,\eta)}d\eta\ e^{-2(|\alpha|+1)\nu t} .
\ea\]
We remark that $\mathscr{D}$ denotes the dissipation terms, $\mathscr{L}$ denotes the linear electric term, $\mathscr{N}$ denotes the nonlinear electric terms, and $\mathscr{R}$ denotes the remainder term.

$\bullet$\underline{\emph{Estimate of $\sL^{T,\alpha}$}.}
We observe that
\[\ba
&\nr{\sL^{T,\alpha}}\ls\sum_k \int\nr{\wh{A^\nu_2\rho}(t,k)}\f{1}{|k|}\nr{D^\alpha_\eta\rr{\eta\wh{\mu}(\eta)}(\bar{\eta})A^\nu_2(t,0,\bar{\eta})}\nr{A^\nu_2D^\alpha_\eta\wh{f}(t,k,\eta)}d\eta\ \emnt e^{-(|\alpha|+1)\nu t} \\
\ls&\sum_k \int\nr{\wh{A^\nu_\f12\rho}(t,k)}|k|^\f12\<t\>^\f32\nr{D^\alpha_\eta\rr{\eta\wh{\mu}(\eta)}(\bar{\eta})A^\nu_2(t,0,\bar{\eta})}\nr{A^\nu_2D^\alpha_\eta\wh{f}(t,k,\eta)}d\eta\ \emnt e^{-(|\alpha|+1)\nu t} \\
\ls&\LLTX{|\na_x|^\f12A^\nu_\f12\rho(t)}\<t\>^\f32\nnr{D^\alpha_\eta\rr{\eta\wh{\mu}(\eta)}A^\nu_2(t,0,\eta)}_{L^2_\eta}e^{-\f{d+2}{2}\nu t}\LLT{A^\nu_2(v^\alpha f)(t)}e^{-(|\alpha|+1)\nu t}\\
\ls&\nnr{|\na_x|^\f12A^\nu_\f12\rho(t)e^{\delta\nu^\f13t}\<\tent\>^\f52}_{L^2_x}\<t\>^{-1}\LLT{A^\nu_2(v^\alpha f)(t)}e^{-(|\alpha|+1)\nu t}\ls\eps\nnr{|\na_x|^\f12A^\nu_\f12\rho(t)e^{\delta\nu^\f13t}\<\tent\>^\f52}_{L^2_x}\<t\>^{-1}.
\ea\]
Thus we get that
\[\int_0^{T^*}|\sL^{T,\alpha}(t)|dt\ls\eps\nnr{|\na_x|^\f12A^\nu_\f12\rho(t)e^{\delta\nu^\f13t}\<\tent\>^\f52}_{L^2([0,T^*],L^2_x)}\ls\eps^2.\]

$\bullet$\underline{\emph{Estimate of $\sN^{T,\alpha}_{top}$}.}
We note that
\[\ba
&\Re\sum_{k,\ell}\int\wh{\rho}(t,\ell)\f{\ell\cdot\bar{\eta}}{|\ell|^2}D^\alpha_\eta\wh{f}(t,k-\ell,\eta-\ell\tent)A^\nu_2(t,k,\eta)A^\nu_2\ov{D^\alpha_\eta\wh{f}(t,k,\eta)}d\eta=\Re\sum_{k,\ell}\int\wh{\rho}(t,\ell)\\
&\times \f{\ell\cdot\bar{\eta}}{|\ell|^2}D^\alpha_\eta\wh{f}(t,k-\ell,\eta-\ell\tent)\rr{A^\nu_2(t,k,\eta)-A^\nu_2(t,k-\ell,\eta-\ell\tent)}A^\nu_2\ov{D^\alpha_\eta\wh{f}(t,k,\eta)}d\eta.
\ea\]
We use the paraproduct decomposition as before to split the integral into two parts: $\<\ell,\ell\tentau\>\leq6\<k-\ell,\eta-\ell\tentau\>$ and $\<\ell,\ell\tentau\>\leq6\<k-\ell,\eta-\ell\tentau\>$. We denote these parts as $\sN^{T,\alpha}_{top,LH}$ and $\sN^{T,\alpha}_{top,HL}$, respectively.
 For LH part, we notice that
\[\ba
&\nr{A^\nu_2(t,k,\eta)-A^\nu_2(t,k-\ell,\eta-\ell\tent)}\ls\<\ell,\ell\tent\>\\
&\times \rr{\<k,\eta\>^{1-s}+\<k-\ell,\eta-\ell\tent\>^{1-s}}^{-1}\rr{A^\nu_2(t,k,\eta)+A^\nu_2(t,k-\ell,\eta-\ell\tent)}.
\ea\]
Since $\bar{\eta}(t,k,\eta)=\bar{\eta}(t,k-\ell,\eta-\ell\tent)$, one has
\[|\bar{\eta}|\ls\ent\<k,\eta\>^\f12\<k-\ell,\eta-\ell\tent\>^\f12\<\tent\>.\] This yields that
\[\ba
&|\sN^{T,\alpha}_{top,LH}|\ls\sum_{k,\ell}\int\nr{\wh{A^\nu\rho}(t,\ell)}\<\ell\>^{-\f{d}2-\f32}\<\tent\>^{-5}\ent\<\tent\>\<\ell,\ell\tent\>\<k-\ell,\eta-\ell\tent\>^{\f s2}\\
&\times \nr{\wh{A^\nu_2(v^\alpha f)}(t,k-\ell,\eta-\ell\tent)}\<k,\eta\>^{\f s2}\nr{\wh{A^\nu_2(v^\alpha f)}(t,k,\eta)}d\eta\ e^{-2(|\alpha|+1)\nu t} \\
\ls&\LLT{\<v\>^mA^\nu f_{\neq}(t)e^{\f{\delta_1}{2}\nu^\f13t}}\emnt\<t\>^{-3}\LLT{A^\nu_{2+\f s2}(v^\alpha f)(t)}^2e^{-2(|\alpha|+1)\nu t} \\
\ls&\eps\<t\>^{-3}\emnt\LLT{A^\nu_{2+\f s2}(v^\alpha f)(t)}^2e^{-2(|\alpha|+1)\nu t} .
\ea\]
By choosing $\eps\ls (\lambda_1-\lambda_\infty)b$, this term can be absorbed by $\sD^{T,\alpha}_{\lambda}$. Next we consider the HL part. Using the fact that $|\bar{\eta}|\ls\<k-\ell,\eta-\ell\tent\>\<t\>$, we get that
\[\ba
&|\sN^{T,\alpha}_{top,HL}|\ls\sum_{k,\ell}\int\nr{\wh{A^\nu_\f12\rho}(t,\ell)}|\ell|^\f12\<t\>^\f32\ent\<t\>\nr{\wh{A^\nu(v^\alpha f)}(t,k-\ell,\eta-\ell\tent)}\<k-\ell\>^{-d}\\
&\times \nr{\wh{A^\nu_2(v^\alpha f)}(t,k,\eta)}d\eta\ e^{-2(|\alpha|+1)\nu t} \\
\ls&\LLTX{|\na_x|^\f12A^\nu_\f12\rho(t)}\ent\<t\>^{\f52}\LLT{A^\nu(v^\alpha f)(t)}e^{-(|\alpha|+1)\nu t}\LLT{A^\nu_2(v^\alpha f)(t)}e^{-(|\alpha|+1)\nu t}\\
\ls&\LLTX{|\na_x|^\f12A^\nu_\f12\rho(t)e^{\delta\nu^\f13t}\<\tent\>^{\f72}}\<t\>^{-1}\eps^2.
\ea\]
Thus it holds that
\[\int_0^{T^*}|\sN^{T,\alpha}_{top,HL}(t)|dt\ls\eps^2\nnr{|\na_x|^\f12A^\nu_\f12\rho(t)e^{\delta\nu^\f13t}\<\tent\>^4}_{L^2([0,T^*],L^2_x)}\ls\eps^3.\]
We conclude that 
\beno \int_0^{T^*}|\sN^{T,\alpha}_{top}(t)|dt\ls \eps\int_0^{T^*}\sD^{T,\alpha}_{\lambda}d\tau+\eps^3.  \eeno

$\bullet$\underline{\emph{Estimate of $\sN^{T,\alpha}_{re}$}.} Using  Proposition \ref{nut} and the similar argument applied in the before, we get that

\[\ba
&|\sN^{T,\alpha}_{re,LH}|\ls\sum_{\substack{\alpha'<\alpha\\|\alpha'|=|\alpha|-1}}C_{\alpha}^{\alpha'}\sum_{k,\ell}\int\nr{\wh{A^\nu\rho}(t,\ell)}\<\ell\>^{-d}\<\tent\>^{-2}\nr{\wh{A^\nu_{2}(v^{\alpha'} f)}(t,k-\ell,\eta-\ell\tent)}\\
&\times \nr{\wh{A^\nu_2(v^\alpha f)}(t,k,\eta)}d\eta\ e^{-(2|\alpha|+1)\nu t} \\
\ls&\sum_{\substack{\alpha'<\alpha\\|\alpha'|=|\alpha|-1}}C_{\alpha}^{\alpha'}\LLTX{A^\nu\rho(t)}\<\tent\>^{-2}\LLT{A^\nu_{2}(v^{\alpha'} f)(t)}e^{-(|\alpha'|+1)\nu t}\LLT{A^\nu_2(v^\alpha f)(t)}\\
&\times e^{-(|\alpha|+1)\nu t}
\ls_m\LLT{\<v\>^mA^\nu f_{\neq}(t)e^{\f{\delta_1}{2}\nu^\f13 t}}\<t\>^{-2}\eps^2\ls\eps^3\<t\>^{-2},
\ea\]
 and
\[\ba
&|\sN^{T,\alpha}_{re,HL}|\ls\sum_{\substack{\alpha'<\alpha\\|\alpha'|=|\alpha|-1}}C_{\alpha}^{\alpha'}\sum_{k,\ell}\int\nr{\wh{A^\nu_\f12\rho}(t,\ell)}|\ell|^\f12\<t\>^\f32\nr{\wh{A^\nu(v^{\alpha'} f)}(t,k-\ell,\eta-\ell\tent)}\<k-\ell\>^{-d}\\
&\times \nr{\wh{A^\nu_2(v^\alpha f)}(t,k,\eta)}d\eta\ e^{-(2|\alpha|+1)\nu t} \\
\ls&\sum_{\substack{\alpha'<\alpha\\|\alpha'|=|\alpha|-1}}C_{\alpha}^{\alpha'}\LLTX{|\na_x|^\f12A^\nu_\f12\rho(t)e^{\delta\nu^\f13t}\<\tent\>^\f52}\<t\>^{-1}\LLT{A^\nu(v^{\alpha'} f)(t)}e^{-(|\alpha'|+1)\nu t}\LLT{A^\nu_2(v^\alpha f)(t)}e^{-(|\alpha|+1)\nu t}\\
\ls_m&\LLTX{|\na_x|^\f12A^\nu_\f12\rho(t)e^{\delta\nu^\f13t}\<\tent\>^\f52}\<t\>^{-1}\eps^2.
\ea\]

We conclude that 
\beno \int_0^{T^*}|\sN^{T,\alpha}_{re}(t)|dt\ls  \eps^3.  \eeno

$\bullet$\underline{\emph{Estimate of $\sR^{T,\alpha}_{re}$}.} It is not difficult to see that
\[\ba
&|\sR^{T,\alpha}_{re}|\le C\nu\sum_{\substack{\alpha'<\alpha\\|\alpha'|=|\alpha|-1}}C_{\alpha}^{\alpha'}\LLT{A^\nu_{2}(v^{\alpha'} f)(t)}e^{-(|\alpha'|+1)\nu t}\<t\>^{-4}\LLT{\vect A^\nu_{2}(v^{\alpha} f)(t)}e^{-(|\alpha|+1)\nu t}\<t\>^{-4}\\
&+C\nu\sum_{\substack{\alpha'<\alpha\\|\alpha'|=|\alpha|-2}}C_{\alpha}^{\alpha'}\LLT{A^\nu_{2}(v^{\alpha'} f)(t)}e^{-(|\alpha'|+1)\nu t}\<t\>^{-4}\LLT{A^\nu_{2}(v^{\alpha} f)(t)}e^{-(|\alpha|+1)\nu t}\<t\>^{-4}\\
\le&\f14\nu\LLT{\vect A^\nu_{2}(v^{\alpha} f)(t)}^2e^{-2(|\alpha|+1)\nu t} +\f14\nu\LLT{A^\nu_{2}(v^{\alpha} f)(t)}^2e^{-2(|\alpha|+1)\nu t} \\
&+C_m\nu\sum_{\substack{\alpha'<\alpha\\|\alpha'|\geq|\alpha|-2}}\LLT{A^\nu_{2}(v^{\alpha'} f)(t)}^2e^{-2(|\alpha'|+1)\nu t} .
\ea\] 

We remark that the first term and the second term can be absorbed by $\sD^{T,\alpha}$ and by $\sD^{T,\alpha}_{|\alpha|}$ respectively. By choosing $K_{|\alpha|}\gg K_{|\alpha|-1}$, the third term can be absorbed by $\sD^{T,\alpha'}_{|\alpha'|}$ for $\alpha'<\alpha$ and $|\alpha'|\geq|\alpha|-2$.

\bigskip
In conclusion, we get that
\ben\label{improvef}
\sup_{0\leq t\leq T^*}E^T(t)\leq E^T(0)+\int_0^{T^*}\nr{\f{d}{dt}E^T(t)}dt
\leq C_3(m,d)\eps^2+C_4(K_{\rho})\eps^2\notag\\
+C_5(K_{\rho }, K_f, K_{ED}, K_{SH})\eps^3.
\een
We remark that the first term comes from the initial data and the second term comes from the linear term.
\section{Enhanced dissipation estimate}
The main purpose of the section is devoted to improving \eqref{Hed}. To do it, we choose
\[\ba
E^{ED}(t):=\sum_{|\alpha|\leq m}\f{e^{-2(|\alpha|+1)\nu t}}{K_{|\alpha|}}\sum_{i=1}^{d}\left(\LLT{\na_x\pa_{x_i}A^\nu(v^\alpha f)(t)}^2+\smc\nu^\f23\LLT{\vect\pa_{x_i}A^\nu(v^\alpha f)(t)}^2\right.\\
\left.+\smc\nu^\f13\iint\na_x\pa_{x_i}A^\nu(v^\alpha f)(t)\cdot\vect\pa_{x_i}A^\nu(v^\alpha f)(t)dvdx\right) e^{2\delta_1\nu^\f13t} 
\ea\]
as the energy functional. For fixed small $\smc$, it's easy to see that
\[\ba&\LLT{\na_x\pa_{x_i}A^\nu(v^\alpha f)(t)}^2+\smc\nu^\f23\LLT{\vect\pa_{x_i}A^\nu(v^\alpha f)(t)}^2+\smc\nu^\f13\iint\na_x\pa_{x_i}A^\nu(v^\alpha f)(t)\cdot\vect\pa_{x_i}A^\nu(v^\alpha f)(t)dvdx\\
&\sim \LLT{\na_x\pa_{x_i}A^\nu(v^\alpha f)(t)}^2+\smc\nu^\f23\LLT{\vect\pa_{x_i}A^\nu(v^\alpha f)(t)}^2.\ea\]
Let us give the estimates term by term.

\subsection{Estimate of $x$-derivative}
By direct computation, we get that
\[\ba
&\f12\f{d}{dt}\rr{\LLT{\na_x\pa_{x_i}A^\nu(v^\alpha f)(t)}^2e^{-2(|\alpha|+1)\nu t}e^{2\delta_1\nu^\f13t} }+\sD^{ \alpha,i,x}+\sD^{ \alpha,i,x}_{|\alpha|}+\sD^{ \alpha,i,x}_\lambda\\
=&\sR^{ \alpha,i,x}_{\delta_1}+\sL^{ \alpha,i,x}+\sN^{ \alpha,i,x}_{top}+\sN^{ \alpha,i,x}_{re}+\sR^{ \alpha,i,x}_{re}.
\ea\]
The terms are defined as follows:
\[\ba
&\sD^{ \alpha,i,x}:=\nu\LLT{\vect\na_x\pa_{x_i}A^\nu(v^\alpha f)(t)}^2e^{-2(|\alpha|+1)\nu t}e^{2\delta_1\nu^\f13t} ,\\
&\sD^{ \alpha,i,x}_{|\alpha|}:=\nu(|\alpha|+1)\LLT{\na_x\pa_{x_i}A^\nu(v^\alpha f)(t)}^2e^{-2(|\alpha|+1)\nu t}e^{2\delta_1\nu^\f13t} ,\\
&\sD^{ \alpha,i,x}_\lambda:=\f{\lambda_1-\lambda_{\infty}}{8}b(1+\tent)^{-b-1}\emnt\LLT{\na_x\pa_{x_i}A^\nu_{\f s2}(v^\alpha f)(t)}^2e^{-2(|\alpha|+1)\nu t}e^{2\delta_1\nu^\f13t}\\
&+\f{\lambda_1-\lambda_{\infty}}{8}b\emnt\LLT{(1+\tent\<\na\>^{s-1})^{-\f{b}{2}-\f12}\na_x\pa_{x_i}A^\nu_{s-\f12}(v^\alpha f)(t)}^2e^{-2(|\alpha|+1)\nu t}e^{2\delta_1\nu^\f13t}
\ea\]
\[\ba
&\sR^{ \alpha,i,x}_{\delta_1}:=\delta_1\nu^\f13\LLT{\na_x\pa_{x_i}A^\nu(v^\alpha f)(t)}^2e^{-2(|\alpha|+1)\nu t}e^{2\delta_1\nu^\f13t} ,\\
&\sL^{ \alpha,i,x}:=-\sum_k \int\wh{\rho}(t,k)\f{k\cdot}{|k|^2}D^\alpha_\eta\rr{\eta\wh{\mu}(\eta)}(\bar{\eta})A^\nu(t,k,\eta)|k|^2k_i^2A^\nu\ov{D^\alpha_\eta\wh{f}(t,k,\eta)}d\eta\ \emnt e^{-(|\alpha|+1)\nu t}e^{2\delta_1\nu^\f13t} ,\\
&\sN^{ \alpha,i,x}_{top}:=-\sum_{k,\ell}\int\wh{\rho}(t,\ell)\f{\ell\cdot\bar{\eta}}{|\ell|^2}D^\alpha_\eta\wh{f}(t,k-\ell,\eta-\ell\tent)A^\nu(t,k,\eta)|k|^2k_i^2A^\nu\ov{D^\alpha_\eta\wh{f}(t,k,\eta)}d\eta\ e^{-2(|\alpha|+1)\nu t}e^{2\delta_1\nu^\f13t} ,\\
&\sN^{ \alpha,i,x}_{re}:=-\sum_{\substack{\alpha'<\alpha\\|\alpha'|=|\alpha|-1}}C_{\alpha}^{\alpha'}\sum_{k,\ell}\int\wh{\rho}(t,\ell)\f{\ell_{\alpha-\alpha'}}{|\ell|^2}D^{\alpha'}_\eta\wh{f}(t,k-\ell,\eta-\ell\tent)A^\nu(t,k,\eta)\\
&\times|k|^2k_i^2A^\nu\ov{D^\alpha_\eta\wh{f}(t,k,\eta)}d\eta\ e^{-(2|\alpha|+1)\nu t}e^{2\delta_1\nu^\f13t} ,\\
&\sR^{ \alpha,i,x}_{re}:=\nu\sum_{\substack{\alpha'<\alpha\\|\alpha'|\geq|\alpha|-2}}C_{\alpha}^{\alpha'}\sum_k\int D^{\alpha-\alpha'}_\eta\rr{\bar{\eta}^2}A^\nu D^{\alpha'}_\eta\wh{f}(t,k,\eta)|k|^2k_i^2A^\nu\ov{D^\alpha_\eta\wh{f}(t,k,\eta)}d\eta\ e^{-2(|\alpha|+1)\nu t}e^{2\delta_1\nu^\f13t} .
\ea\]
We remark that the time derivative acting on $e^{\delta_1\nu^\f13t}$ yields $\sR^{ \alpha,i,x}_{\delta_1}$, and $D_\eta^\alpha$ acting on $\bar{\eta}^2$ yields $\sR^{ \alpha,i,x}_{re}$.
We note that if we choose $\delta_1<\f14\smc$, then $\sR^{ \alpha,i,x}_{\delta_1}$ can be absorbed by $\sD^{ \alpha,i,hc}$.

$\bullet$\underline{\emph{Estimate of $\sL^{ \alpha,i,x}$}.}
Since $\delta_1<\delta$, we infer that
\[\ba
&\nr{\sL^{ \alpha,i,x}}\ls\sum_k \int\nr{\wh{A^\nu_\f12\rho}(t,k)}|k|^\f12\nr{D^\alpha_\eta\rr{\eta\wh{\mu}(\eta)}(\bar{\eta})A^\nu(t,0,\bar{\eta})}|k||k_i|\nr{A^\nu D^\alpha_\eta\wh{f}(t,k,\eta)}d\eta\\
&\times \emnt e^{-(|\alpha|+1)\nu t}e^{2\delta_1\nu^\f13t} \ls\nnr{|\na_x|^\f12A^\nu_\f12\rho(t)e^{\delta\nu^\f13t}\<\tent\>}_{L^2_x}\<t\>^{-1}\LLT{\na_x\pa_{x_i}A^\nu(v^\alpha f)(t)}\\
&\times e^{-(|\alpha|+1)\nu t}e^{\delta_1\nu^\f13t}\ls\eps\nnr{|\na_x|^\f12A^\nu_\f12\rho(t)e^{\delta\nu^\f13t}\<\tent\>}_{L^2_x}\<t\>^{-1}.
\ea\]
Thus
\[\int_0^{T^*}|\sL^{ \alpha,i,x}(t)|dt\ls\eps\nnr{|\na_x|^\f12A^\nu_\f12\rho(t)e^{\delta\nu^\f13t}\<\tent\>}_{L^2([0,T^*],L^2_x)}\ls\eps^2.\]

$\bullet$\underline{\emph{Estimate of $\sN^{ \alpha,i,x}_{top}$}.} Thanks to the fact that
\[|k|^2\ k_i^2=\ell\cdot k\ \ell_ik_i+(k-\ell)\cdot k\ \ell_ik_i+\ell\cdot k\ (k_i-\ell_i)k_i+(k-\ell)\cdot k\ (k_i-\ell_i)k_i,\]
we split $\sN^{ \alpha,i,x}_{top}$ into four parts: $\sN^{ \alpha,i,x}_{top,1}, \sN^{ \alpha,i,x}_{top,2}, \sN^{ \alpha,i,x}_{top,3}, \sN^{ \alpha,i,x}_{top,4}$. Again by paraproduct decomposition, we split the frequency into LH part and HL part. We have
\[\ba
&\nr{\sN^{ \alpha,i,x}_{top,1,LH}}\ls\sum_{k,\ell}\int\nr{\wh{A^\nu\rho}(t,\ell)}\<\ell\>^{-\f{d}{2}-\f12}\<\tent\>^{-4}\ent\<\tent\>\nr{\wh{A^\nu_1(v^\alpha f)}(t,k-\ell,\eta-\ell\tent)}\\
&\times |k||k_i|\nr{\wh{A^\nu(v^\alpha f)}(t,k,\eta)}d\eta\ e^{-2(|\alpha|+1)\nu t}e^{2\delta_1\nu^\f13t} \\
\ls&\LLTX{A^\nu\rho(t)e^{\delta_1\nu^\f13t}}\<\tent\>^{-3}\ent\LLT{A^\nu_1(v^\alpha f)(t)}e^{-(|\alpha|+1)\nu t}\LLT{\na_x\pa_{x_i}A^\nu(v^\alpha f)(t)}\\
&\times e^{-(|\alpha|+1)\nu t}e^{\delta_1\nu^\f13t}\ls\eps^2\LLTX{A^\nu\rho(t)e^{\delta\nu^\f13t}}\<t\>^{-3},
\ea\] 
and
\[\ba
&\nr{\sN^{ \alpha,i,x}_{top,1,HL}}\ls\sum_{k,\ell}\int\nr{\wh{A^\nu_{\f12}\rho}(t,\ell)}|\ell|^\f12\ent\<t\>\nr{\wh{A^\nu(v^\alpha f)}(t,k-\ell,\eta-\ell\tent)}\<k-\ell\>^{-d}\\
&\times |k||k_i|\nr{\wh{A^\nu(v^\alpha f)}(t,k,\eta)}d\eta\ e^{-2(|\alpha|+1)\nu t}e^{2\delta_1\nu^\f13t} \\
\ls&\LLTX{|\na_x|^\f12 A^\nu_{\f12}\rho(t)e^{\delta\nu^\f13t}\<\tent\>^3}\<t\>^{-2}\LLT{A^\nu(v^\alpha f)(t)}e^{-(|\alpha|+1)\nu t}\LLT{\na_x\pa_{x_i}A^\nu(v^\alpha f)(t)}\\
&\times e^{-(|\alpha|+1)\nu t}e^{\delta_1\nu^\f13t}\ls\eps^2\LLTX{|\na_x|^\f12 A^\nu_{\f12}\rho(t)e^{\delta\nu^\f13t}\<\tent\>^3}\<t\>^{-2}.
\ea\]
Similarly for $\sN^{ \alpha,i,x}_{top,2}$, we have
\[\ba
&|\sN^{ \alpha,i,x}_{top,2}|\ls\LLTX{A^\nu\rho(t)e^{\delta\nu^\f13t}}\<t\>^{-2}\LLT{A^\nu_2(v^\alpha f)(t)}e^{-(|\alpha|+1)\nu t}\LLT{\na_x\pa_{x_i}A^\nu(v^\alpha f)(t)}\\
&\times e^{-(|\alpha|+1)\nu t}e^{\delta_1\nu^\f13t}+\LLTX{|\na_x|^\f12 A^\nu_{\f12}\rho(t)e^{\delta\nu^\f13t}\<\tent\>^4}\<t\>^{-2}\LLT{A^\nu(v^\alpha f)(t)}e^{-(|\alpha|+1)\nu t}\\
&\times \LLT{\na_x\pa_{x_i}A^\nu(v^\alpha f)(t)}e^{-(|\alpha|+1)\nu t}e^{\delta_1\nu^\f13t}\ls\eps^2\LLTX{|\na_x|^\f12 A^\nu_{\f12}\rho(t)e^{\delta\nu^\f13t}\<\tent\>^{4}}\<t\>^{-2}.
\ea\]
It is easy to see that $\sN^{ \alpha,i,x}_{top,3}$ satisfies the same estimate as $\sN^{ \alpha,i,x}_{top,2}$. For $\sN^{ \alpha,i,x}_{top,4}$, we observe that
\[\ba
&\Re\sum_{k,\ell}\int\wh{\rho}(t,\ell)\f{\ell\cdot\bar{\eta}}{|\ell|^2}(k-\ell)(k_i-\ell_i)D^\alpha_\eta\wh{f}(t,k-\ell,\eta-\ell\tent)A^\nu(t,k,\eta)\cdot k\ k_i\\
&\times A^\nu\ov{D^\alpha_\eta\wh{f}(t,k,\eta)}d\eta=\Re\sum_{k,\ell}\int\wh{\rho}(t,\ell)\f{\ell\cdot\bar{\eta}}{|\ell|^2}(k-\ell)(k_i-\ell_i)D^\alpha_\eta\wh{f}(t,k-\ell,\eta-\ell\tent)\\
&\rr{A^\nu(t,k,\eta)-A^\nu(t,k-\ell,\eta-\ell\tent)}k\ k_iA^\nu\ov{D^\alpha_\eta\wh{f}(t,k,\eta)}d\eta.
\ea\]
For the LH part, we note that
\ben\label{commu}
\nr{A^\nu(t,k,\eta)-A^\nu(t,k-\ell,\eta-\ell\tent)}\ls\<\ell,\ell\tent\>\notag\\
\times \rr{\<k,\eta\>^{1-s}+\<k-\ell,\eta-\ell\tent\>^{1-s}}^{-1}\rr{A^\nu(t,k,\eta)+A^\nu(t,k-\ell,\eta-\ell\tent)}.
\een
This yields that
\[\ba
&|\sN^{ \alpha,i,x}_{top,4,LH}|\ls\sum_{k,\ell}\int\nr{\wh{A^\nu\rho}(t,\ell)}\<\ell\>^{-\f{d}{2}-\f32}\<\tent\>^{-5}\ent\<\tent\>\<\ell,\ell\tent\>|k-\ell||k_i-\ell_i|\\
&\times \<k-\ell,\eta-\ell\tent\>^{\f s2}\nr{\wh{A^\nu(v^\alpha f)}(t,k-\ell,\eta-\ell\tent)}|k||k_i|\<k,\eta\>^{\f s2}\nr{\wh{A^\nu(v^\alpha f)}(t,k,\eta)}d\eta\\&\times  e^{-2(|\alpha|+1)\nu t}e^{2\delta_1\nu^\f13t} \\
\ls&\<\tent\>^{-3}\emnt\LLT{\<v\>^mA^\nu f_{\neq}(t)e^{\f{\delta_1}{2}\nu^\f13 t}}\LLT{\na_x\pa_{x_i}A^\nu_{\f s2}(v^\alpha f)(t)}^2e^{-2(|\alpha|+1)\nu t}e^{2\delta_1\nu^\f13t} \\
\ls&\eps\<\tent\>^{-3}\emnt\LLT{\na_x\pa_{x_i}A^\nu_{\f s2}(v^\alpha f)(t)}^2e^{-2(|\alpha|+1)\nu t}e^{2\delta_1\nu^\f13t} .
\ea\]
By choosing $\eps\ls(\lambda_1-\lambda_\infty)b$, this term can be absorbed by $\sD^{ \alpha,i,x}_{\lambda}$. For the HL part, we get that
\[\ba
&|\sN^{ \alpha,i,x}_{top,4,HL}|\ls\LLTX{|\na_x|^\f12 A^\nu_{\f12}\rho(t)e^{\delta\nu^\f13t}\<\tent\>^4}\<t\>^{-1}\LLT{A^\nu(v^\alpha f)(t)}e^{-(|\alpha|+1)\nu t}\\
&\times \LLT{\na_x\pa_{x_i}A^\nu(v^\alpha f)(t)}e^{-(|\alpha|+1)\nu t}e^{\delta_1\nu^\f13t}\ls\eps^2\LLTX{|\na_x|^\f12 A^\nu_{\f12}\rho(t)e^{\delta\nu^\f13t}\<\tent\>^4}\<t\>^{-1}.
\ea\]

$\bullet$\underline{\emph{Estimate of $\sN^{ \alpha,i,x}_{re}$ and $\sR^{ \alpha,i,x}_{re}$}.} For $\sN^{ \alpha,i,x}_{re}$, it holds that
\[\ba
&|\sN^{ \alpha,i,x}_{re}|\ls\sum_{\substack{\alpha'<\alpha\\|\alpha'|=|\alpha|-1}}C_{\alpha}^{\alpha'}\left(\LLTX{A^\nu\rho(t)e^{\delta\nu^\f13t}}\<t\>^{-2}\LLT{A^\nu_2(v^{\alpha'} f)(t)}e^{-(|\alpha'|+1)\nu t}\right.\\
&\times \LLT{\na_x\pa_{x_i}A^\nu(v^\alpha f)(t)}e^{-(|\alpha|+1)\nu t}e^{\delta_1\nu^\f13t}+\LLTX{|\na_x|^\f12 A^\nu_{\f12}\rho(t)e^{\delta\nu^\f13t}\<\tent\>^4}\<t\>^{-1}\LLT{A^\nu(v^{\alpha'} f)(t)}\\
&\times \left.e^{-(|\alpha'|+1)\nu t}\LLT{\na_x\pa_{x_i}A^\nu(v^\alpha f)(t)}e^{-(|\alpha|+1)\nu t}e^{\delta_1\nu^\f13t}
\right)\ls\eps^2\LLTX{|\na_x|^\f12 A^\nu_{\f12}\rho(t)e^{\delta\nu^\f13t}\<\tent\>^4}\<t\>^{-1}.
\ea\]
For $\sR^{ \alpha,i,x}_{re}$, it holds that
\[\ba
&|\sR^{ \alpha,i,x}_{re}|\leq\f14\nu\LLT{\vect \na_x\pa_{x_i}A^\nu(v^{\alpha} f)(t)}^2e^{-2(|\alpha|+1)\nu t}e^{2\delta_1\nu^\f13t} +\f14\nu\LLT{\na_x\pa_{x_i}A^\nu(v^{\alpha} f)(t)}^2\\
&\times e^{-2(|\alpha|+1)\nu t}e^{2\delta_1\nu^\f13t} +C_m\nu\sum_{\substack{\alpha'<\alpha\\|\alpha'|\geq|\alpha|-2}}\LLT{\na_x\pa_{x_i}A^\nu(v^{\alpha'} f)(t)}^2e^{-2(|\alpha'|+1)\nu t}e^{2\delta_1\nu^\f13t} .
\ea\]

We remark that the first and the second terms can be absorbed by $\sD^{ \alpha,i,x}$ and by $\sD^{ \alpha,i,x}_{|\alpha|}$ respectively. The third term can be absorbed by $\sD^{ \alpha',i,x}_{|\alpha'|}$ for $\alpha'<\alpha$ and $|\alpha'|\geq|\alpha|-2$.

\subsection{Estimate of $v$-derivative}
By direct computation, we derive that 
\[\ba
&\f12\f{d}{dt}\rr{\smc\nu^\f23\LLT{\vect\pa_{x_i}A^\nu(v^\alpha f)(t)}^2e^{-2(|\alpha|+1)\nu t}e^{2\delta_1\nu^\f13t} }+\sD^{ \alpha,i,v}+\sD^{ \alpha,i,v}_{|\alpha|}+\sD^{ \alpha,i,v}_\lambda\\
=&\sR^{ \alpha,i,v}_{\delta_1}+\sR^{ \alpha,i,v}_{v,1}+\sR^{ \alpha,i,v}_{v,2}+\sL^{ \alpha,i,v}+\sN^{ \alpha,i,v}_{top}+\sN^{ \alpha,i,v}_{re}+\sR^{ \alpha,i,v}_{re}.
\ea\]
The terms are defined as follows: 
\[\ba
&\sD^{ \alpha,i,v}:=\smc\nu^\f53\LLT{\rr{\vect}^2\pa_{x_i}A^\nu(v^\alpha f)(t)}^2e^{-2(|\alpha|+1)\nu t}e^{2\delta_1\nu^\f13t} ,\\
&\sD^{ \alpha,i,v}_{|\alpha|}:=\smc\nu^\f53(|\alpha|+1)\LLT{\vect\pa_{x_i}A^\nu(v^\alpha f)(t)}^2e^{-2(|\alpha|+1)\nu t}e^{2\delta_1\nu^\f13t} ,\\
&\sD^{ \alpha,i,v}_\lambda:=\f{\lambda_1-\lambda_{\infty}}{8}b(1+\tent)^{-b-1}\emnt\LLT{\vect\pa_{x_i}A^\nu_{\f s2}(v^\alpha f)(t)}^2e^{-2(|\alpha|+1)\nu t}e^{2\delta_1\nu^\f13t}\\
&+\f{\lambda_1-\lambda_{\infty}}{8}b\emnt\LLT{(1+\tent\<\na\>^{s-1})^{-\f{b}{2}-\f12}\vect\pa_{x_i}A^\nu_{s-\f12}(v^\alpha f)(t)}^2e^{-2(|\alpha|+1)\nu t}e^{2\delta_1\nu^\f13t}
\ea\]
\[\ba
&\sR^{ \alpha,i,v}_{\delta_1}:=\delta_1\smc\nu\LLT{\vect\pa_{x_i}A^\nu(v^\alpha f)(t)}^2e^{-2(|\alpha|+1)\nu t}e^{2\delta_1\nu^\f13t} ,\\
&\sR^{ \alpha,i,v}_{v,1}:=\smc\nu^\f53\LLT{\vect\pa_{x_i}A^\nu(v^\alpha f)(t)}^2e^{-2(|\alpha|+1)\nu t}e^{2\delta_1\nu^\f13t} ,\\
&\sR^{ \alpha,i,v}_{v,2}:=-\smc\nu^\f23\iint\na_x\pa_{x_i}A^\nu(v^\alpha f)(t)\cdot\vect\pa_{x_i}A^\nu(v^\alpha f)(t)dvdx\ e^{-2(|\alpha|+1)\nu t}e^{2\delta_1\nu^\f13t} , \\
&\sL^{ \alpha,i,v}:=-\smc\nu^\f23\sum_k \int\wh{\rho}(t,k)\f{k\cdot}{|k|^2}D^\alpha_\eta\rr{\eta\mu(\eta)}(\bar{\eta})A^\nu(t,k,\eta)|\bar{\eta}|^2k_i^2A^\nu\ov{D^\alpha_\eta\wh{f}(t,k,\eta)}d\eta\ \emnt e^{-(|\alpha|+1)\nu t}e^{2\delta_1\nu^\f13t} ,\\
&\sN^{ \alpha,i,v}_{top}:=-\smc\nu^\f23\sum_{k,\ell}\int\wh{\rho}(t,\ell)\f{\ell\cdot\bar{\eta}}{|\ell|^2}D^\alpha_\eta\wh{f}(t,k-\ell,\eta-\ell\tent)A^\nu(t,k,\eta)|\bar{\eta}|^2k_i^2A^\nu\ov{D^\alpha_\eta\wh{f}(t,k,\eta)}d\eta\ e^{-2(|\alpha|+1)\nu t} e^{2\delta_1\nu^\f13t} ,\\
&\sN^{ \alpha,i,v}_{re}:=-\smc\nu^\f23\sum_{\substack{\alpha'<\alpha\\|\alpha'|=|\alpha|-1}}C_{\alpha}^{\alpha'}\sum_{k,\ell}\int\wh{\rho}(t,\ell)\f{\ell_{\alpha-\alpha'}}{|\ell|^2}D^{\alpha'}_\eta\wh{f}(t,k-\ell,\eta-\ell\tent)A^\nu(t,k,\eta)\\
&\times |\bar{\eta}|^2k_i^2A^\nu\ov{D^\alpha_\eta\wh{f}(t,k,\eta)}d\eta\ e^{-(2|\alpha|+1)\nu t}e^{2\delta_1\nu^\f13t} ,\\
&\sR^{ \alpha,i,v}_{re}:=\smc\nu^\f53\sum_{\substack{\alpha'<\alpha\\|\alpha'|\geq|\alpha|-2}}C_{\alpha}^{\alpha'}\sum_k\int D^{\alpha-\alpha'}_\eta\rr{\bar{\eta}^2}A^\nu D^{\alpha'}_\eta\wh{f}(t,k,\eta)|\bar{\eta}|^2k_i^2A^\nu\ov{D^\alpha_\eta\wh{f}(t,k,\eta)}d\eta\ e^{-2(|\alpha|+1)\nu t}e^{2\delta_1\nu^\f13t} .
\ea\]
We remark that the time derivative acting on $\vect$ yields $\sR^{ \alpha,i,v}_{v,1}$ and $\sR^{ \alpha,i,v}_{v,2}$.
We shall estimate the terms in the right-hand side one by one.
One may check that $\sR^{ \alpha,i,v}_{\delta_1}, \sR^{ \alpha,i,v}_{v,1}$ can be absorbed by $\sD^{ \alpha,i,x}$. For $\sR^{ \alpha,i,v}_{v,2}$,
\[\ba
&|\sR^{ \alpha,i,v}_{v,2}|\leq\smc\nu^\f23\LLT{\na_x\pa_{x_i}A^\nu(v^\alpha f)(t)}\LLT{\vect\pa_{x_i}A^\nu(v^\alpha f)(t)}e^{-2(|\alpha|+1)\nu t}e^{2\delta_1\nu^\f13t} \\
\leq&\f18\smc\nu^\f13\LLT{\na_x\pa_{x_i}A^\nu(v^\alpha f)(t)}^2e^{-2(|\alpha|+1)\nu t}e^{2\delta_1\nu^\f13t} +8\smc\nu\LLT{\vect\pa_{x_i}A^\nu(v^\alpha f)(t)}^2e^{-2(|\alpha|+1)\nu t}e^{2\delta_1\nu^\f13t} .
\ea\]
The first and second terms can be absorbed by $\sD^{ \alpha,i,hc}$ and  by $\sD^{ \alpha,i,x}$ respectively.

$\bullet$ \underline{\emph{Estimate of $\sL^{ \alpha,i,v}$}.} It holds that
\[\ba
&|\sL^{ \alpha,i,v}|\ls\sqrt{\smc}\nu^\f13\LLT{A^\nu\rho(t)e^{\delta\nu^\f13t}\<\tent\>^3}\<t\>^{-3}\sqrt{\smc}\nu^\f13\LLT{\vect\pa_{x_i}A^\nu(v^\alpha f)(t)}e^{-(|\alpha|+1)\nu t}e^{\delta_1\nu^\f13t}\\
\ls&\sqrt{\smc}\nu^\f13\LLT{A^\nu\rho(t)e^{\delta\nu^\f13t}\<\tent\>^3}\<t\>^{-3}\eps,
\ea\]
which implies that   $\int_0^{T^*}|\sL^{ \alpha,i,v}|dt\ls\nu^{\f13}\eps^2$.

$\bullet$ \underline{\emph{Estimate of $\sN^{ \alpha,i,v}_{top}$}.} We observe that $k_i^2=\ell_ik_i+(k_i-\ell_i)k_i$.  From this, we split  $\sN^{ \alpha,i,v}_{top}$ into two parts: $\sN^{ \alpha,i,v}_{top,1}, \sN^{ \alpha,i,v}_{top,2}$. We first estimate  $\sN^{ \alpha,i,v}_{top,1,LH}$.
\[\ba
&\nr{\sN^{ \alpha,i,v}_{top,1,LH}}\ls\smc\nu^\f23\sum_{k,\ell}\int\nr{\wh{A^\nu\rho}(t,\ell)}\<\ell\>^{-\f{d}{2}-\f12}\<\tent\>^{-5}e^{2\nu t}\<\tent\>^2\nr{\wh{A^\nu_2(v^\alpha f)}(t,k-\ell,\eta-\ell\tent)}\\
&\times |\bar{\eta}||k_i|\nr{\wh{A^\nu(v^\alpha f)}(t,k,\eta)}d\eta\ e^{-2(|\alpha|+1)\nu t}e^{2\delta_1\nu^\f13t} \\
\ls&\sqrt{\smc}\nu^\f13\LLTX{A^\nu\rho(t)e^{\delta\nu^\f13t}}\<t\>^{-3}\LLT{A^\nu_2(v^\alpha f)(t)}e^{-(|\alpha|+1)\nu t}\sqrt{\smc}\nu^\f13\LLT{\vect\pa_{x_i}A^\nu(v^\alpha f)(t)}\\
&\times e^{-(|\alpha|+1)\nu t}e^{\delta_1\nu^\f13t}\ls\eps^2\sqrt{\smc}\nu^\f13\LLTX{A^\nu\rho(t)e^{\delta\nu^\f13t}}\<t\>^{-3}.
\ea\]
Then we estimate $\sN^{ \alpha,i,x}_{top,1,HL}$.
\[\ba
&\nr{\sN^{ \alpha,i,x}_{top,1,HL}}\ls\smc\nu^\f23\sum_{k,\ell}\int\nr{\wh{A^\nu\rho}(t,\ell)}e^{2\nu t}\<t\>^2\nr{\wh{A^\nu(v^\alpha f)}(t,k-\ell,\eta-\ell\tent)}\<k-\ell\>^{-d}\\
&\times |\bar{\eta}||k_i|\nr{\wh{A^\nu(v^\alpha f)}(t,k,\eta)}d\eta\ e^{-2(|\alpha|+1)\nu t}e^{2\delta_1\nu^\f13t} \\
\ls&\sqrt{\smc}\nu^\f13\LLTX{A^\nu\rho(t)e^{\delta\nu^\f13t}\<\tent\>^3}\<t\>^{-1}\LLT{A^\nu(v^\alpha f)(t)}e^{-(|\alpha|+1)\nu t}\sqrt{\smc}\nu^\f13\LLT{\vect\pa_{x_i}A^\nu(v^\alpha f)(t)}\\
&\times e^{-(|\alpha|+1)\nu t}e^{\delta_1\nu^\f13t}\ls\eps^2\LLTX{A^\nu\rho(t)e^{\delta\nu^\f13t}\<\tent\>^3}\<t\>^{-1}.
\ea\]
Next we estimate $\sN^{ \alpha,i,v}_{top,2}$. Using the commutator, we have
\[\ba
&\Re\sum_{k,\ell}\int\wh{\rho}(t,\ell)\f{\ell\cdot\bar{\eta}}{|\ell|^2}\bar{\eta}(k_i-\ell_i)D^\alpha_\eta\wh{f}(t,k-\ell,\eta-\ell\tent)A^\nu(t,k,\eta)\cdot\bar{\eta}\ k_i\\
&\times A^\nu\ov{D^\alpha_\eta\wh{f}(t,k,\eta)}d\eta=\Re\sum_{k,\ell}\int\wh{\rho}(t,\ell)\f{\ell\cdot\bar{\eta}}{|\ell|^2}\bar{\eta}(k_i-\ell_i)D^\alpha_\eta\wh{f}(t,k-\ell,\eta-\ell\tent)\\
&\times \rr{A^\nu(t,k,\eta)-A^\nu(t,k-\ell,\eta-\ell\tent)}\cdot\bar{\eta}\ k_iA^\nu\ov{D^\alpha_\eta\wh{f}(t,k,\eta)}d\eta.
\ea\]
For the LH part, by \eqref{commu}, we get that
\[\ba
&|\sN^{ \alpha,i,v}_{top,2,LH}|\ls\sum_{k,\ell}\smc\nu^\f23\int\nr{\wh{A^\nu\rho}(t,\ell)}\<\ell\>^{-\f{d}{2}-\f32}\<\tent\>^{-5}\ent\<\tent\>\<\ell,\ell\tent\>|\bar{\eta}||k_i-\ell_i|\\
&\times \<k-\ell,\eta-\ell\tent\>^{\f s2}\nr{\wh{A^\nu(v^\alpha f)}(t,k-\ell,\eta-\ell\tent)}|\bar{\eta}||k_i|\<k,\eta\>^{\f s2}\nr{\wh{A^\nu(v^\alpha f)}(t,k,\eta)}d\eta\\&\times  e^{-2(|\alpha|+1)\nu t}e^{2\delta_1\nu^\f13t} \\
\ls&\smc\nu^\f23\<\tent\>^{-3}\emnt\LLTX{\<v\>^mA^\nu f_{\neq} e^{\f{\delta_1}{2}\nu^\f13 t}}\LLT{\vect\pa_{x_i}A^\nu_{\f s2}(v^\alpha f)(t)}^2e^{-2(|\alpha|+1)\nu t}e^{2\delta_1\nu^\f13t} \\
\ls&\eps\smc\nu^\f23\<\tent\>^{-3}\emnt\LLT{\vect\pa_{x_i}A^\nu_{\f s2}(v^\alpha f)(t)}^2e^{-2(|\alpha|+1)\nu t}e^{2\delta_1\nu^\f13t} .
\ea\]
Since  $\eps$  is sufficiently small, this term can be absorbed by $\sD^{ \alpha,i,v}_{\lambda}$. For the HL part, we obtain that
\[\ba
&|\sN^{ \alpha,i,v}_{top,2,HL}|\ls\sqrt{\smc}\nu^\f13\LLTX{A^\nu\rho(t)e^{\delta\nu^\f13t}\<\tent\>^4}\<t\>^{-1}\LLT{A^\nu(v^\alpha f)(t)}e^{-(|\alpha|+1)\nu t}\\
&\times \sqrt{\smc}\nu^\f13\LLT{\vect\pa_{x_i}A^\nu(v^\alpha f)(t)}e^{-(|\alpha|+1)\nu t}e^{\delta_1\nu^\f13t}\ls\eps^2\LLTX{A^\nu\rho(t)e^{\delta\nu^\f13t}\<\tent\>^4}\<t\>^{-1}.
\ea\]
 
$\bullet$ \underline{\emph{Estimate of $\sN^{ \alpha,i,v}_{re}$ and $\sR^{ \alpha,i,v}_{re}$}.} For $\sN^{ \alpha,i,v}_{re}$, it holds that
\[\ba
&|\sN^{ \alpha,i,v}_{re}|\ls\sum_{\substack{\alpha'<\alpha\\|\alpha'|=|\alpha|-1}}C_{\alpha}^{\alpha'}\left(\sqrt{\smc}\nu^\f13\LLTX{A^\nu\rho(t)e^{\delta\nu^\f13t}}\<t\>^{-2}\LLT{A^\nu_2(v^{\alpha'} f)(t)}e^{-(|\alpha'|+1)\nu t}\right.\\
&\times \sqrt{\smc}\nu^\f13\LLT{\vect\pa_{x_i}A^\nu(v^\alpha f)(t)}e^{-(|\alpha|+1)\nu t}e^{\delta_1\nu^\f13t}+\sqrt{\smc}\nu^\f13\LLTX{A^\nu\rho(t)e^{\delta\nu^\f13t}\<\tent\>^4}\<t\>^{-1}\\
&\times \left.\LLT{A^\nu(v^{\alpha'} f)(t)}e^{-(|\alpha'|+1)\nu t}\sqrt{\smc}\nu^\f13\LLT{\vect\pa_{x_i}A^\nu(v^\alpha f)(t)}e^{-(|\alpha|+1)\nu t}e^{\delta_1\nu^\f13t}\<t\>^{-3}\right)\\
&\ls\eps^2\LLTX{A^\nu\rho(t)e^{\delta\nu^\f13t}\<\tent\>^4}\<t\>^{-1}.
\ea\]
For $\sR^{ \alpha,i,v}_{re}$, it holds that
\[\ba
&|\sR^{ \alpha,i,x}_{re}|\leq\f14\smc\nu^\f53\LLT{\rr{\vect}^2\pa_{x_i}A^\nu(v^{\alpha} f)(t)}^2e^{-2(|\alpha|+1)\nu t}e^{2\delta_1\nu^\f13t} +\f14\smc\nu^\f53\LLT{\vect\pa_{x_i}A^\nu(v^{\alpha} f)(t)}^2\\
&\times e^{-2(|\alpha|+1)\nu t}e^{2\delta_1\nu^\f13t} +C_m\smc\nu^\f53\sum_{\substack{\alpha'<\alpha\\|\alpha'|\geq|\alpha|-2}}\LLT{\vect\pa_{x_i}A^\nu(v^{\alpha'} f)(t)}^2e^{-2(|\alpha'|+1)\nu t}e^{2\delta_1\nu^\f13t} .
\ea\]
We emphasize that the first and second terms can be absorbed by $\sD^{ \alpha,i,v}$ and by $\sD^{ \alpha,i,v}_{|\alpha|}$ respectively. The third term can be absorbed by $\sD^{ \alpha',i,v}_{|\alpha'|}$ for $\alpha'<\alpha$ and $|\alpha'|\geq|\alpha|-2$.

\subsection{Hypocoercivity}
Direct computation implies that
\[\ba
&\f12\f{d}{dt}\rr{\smc\nu^\f13\iint\na_x\pa_{x_i}A^\nu(v^\alpha f)(t)\cdot\vect\pa_{x_i}A^\nu(v^\alpha f)(t)dvdx\ e^{-2(|\alpha|+1)\nu t}e^{2\delta_1\nu^\f13t} }\\
&+\sD^{ \alpha,i,hc}=\sR^{ \alpha,i,hc}_{\delta_1}+\sR^{ \alpha,i,hc}_{v}+\sR^{ \alpha,i,hc}_{D}+\sR^{ \alpha,i,hc}_{|\alpha|}+\sR^{ \alpha,i,hc}_\lambda\\
&+\sL^{ \alpha,i,hc}+\sN^{ \alpha,i,hc}_{top}+\sN^{ \alpha,i,hc}_{re}+\sR^{ \alpha,i,hc}_{re}.
\ea\]
The terms are defined as follows:
\[\ba
&\sD^{ \alpha,i,hc}:=\f12\smc\nu^\f13\LLT{\na_x\pa_{x_i}A^\nu(v^\alpha f)(t)}^2e^{-2(|\alpha|+1)\nu t}e^{2\delta_1\nu^\f13t} ,\\
&\sR^{ \alpha,i,hc}_{\delta_1}:=\delta_1\smc\nu^\f23\iint\na_x\pa_{x_i}A^\nu(v^\alpha f)(t)\cdot\vect\pa_{x_i}A^\nu(v^\alpha f)(t)dvdx\ e^{-2(|\alpha|+1)\nu t}e^{2\delta_1\nu^\f13t} ,\\
&\sR^{ \alpha,i,hc}_{v}:=\f12\smc\nu^\f43\iint\na_x\pa_{x_i}A^\nu(v^\alpha f)(t)\cdot\vect\pa_{x_i}A^\nu(v^\alpha f)(t)dvdx\ e^{-2(|\alpha|+1)\nu t}e^{2\delta_1\nu^\f13t} ,\\
&\sR^{ \alpha,i,hc}_{D}:=-\smc\nu^\f43\iint\vect\na_x\pa_{x_i}A^\nu(v^\alpha f)(t)\cdot\rr{\vect}^2\pa_{x_i}A^\nu(v^\alpha f)(t)dvdx\ e^{-2(|\alpha|+1)\nu t}e^{2\delta_1\nu^\f13t} ,\\
&\sR^{ \alpha,i,hc}_{|\alpha|}:=-\smc\nu^\f43(|\alpha|+1)\iint\na_x\pa_{x_i}A^\nu(v^\alpha f)(t)\cdot\vect\pa_{x_i}A^\nu(v^\alpha f)(t)dvdx\ e^{-2(|\alpha|+1)\nu t}e^{2\delta_1\nu^\f13t} ,\\
\ea\]

\[\ba
&\sR^{ \alpha,i,hc}_\lambda:=-\f{\lambda_1-\lambda_\infty}{8}b(1+\tent)^{-b}\emnt\smc\nu^\f13\iint\na_x\pa_{x_i}A^\nu_\f s2(v^\alpha f)(t)\cdot\vect\pa_{x_i}A^\nu_\f s2(v^\alpha f)(t)dvdx\ e^{-2(|\alpha|+1)\nu t}e^{2\delta_1\nu^\f13t}\\
&+-\f{\lambda_1-\lambda_\infty}{8}b\emnt\smc\nu^\f13\iint(1+\tent\<\na\>^{s-1})^{-b-1}\na_x\pa_{x_i}A^\nu_{s-\f12}(v^\alpha f)(t)\cdot\vect\pa_{x_i}A^\nu_{s-\f12}(v^\alpha f)(t)dvdx\ e^{-2(|\alpha|+1)\nu t}e^{2\delta_1\nu^\f13t}\\
&\sL^{ \alpha,i,hc}:=-\smc\nu^\f13\sum_k \int\wh{\rho}(t,k)\f{k\cdot}{|k|^2}D^\alpha_\eta\rr{\eta\wh{\mu}(\eta)}(\bar{\eta})A^\nu(t,k,\eta)k\cdot\bar{\eta}\ k_i^2A^\nu\ov{D^\alpha_\eta\wh{f}(t,k,\eta)}d\eta\ \emnt e^{-(|\alpha|+1)\nu t}e^{2\delta_1\nu^\f13t} ,\\
&\sN^{ \alpha,i,hc}_{top}:=-\smc\nu^\f13\sum_{k,\ell}\int\wh{\rho}(t,\ell)\f{\ell\cdot\bar{\eta}}{|\ell|^2}D^\alpha_\eta\wh{f}(t,k-\ell,\eta-\ell\tent)A^\nu(t,k,\eta)k\cdot\bar{\eta}\ k_i^2A^\nu\ov{D^\alpha_\eta\wh{f}(t,k,\eta)}d\eta\ e^{-2(|\alpha|+1)\nu t}\\
&\times e^{2\delta_1\nu^\f13t} ,\\
&\sN^{ \alpha,i,hc}_{re}:=-\smc\nu^\f13\sum_{\substack{\alpha'<\alpha\\|\alpha'|=|\alpha|-1}}C_{\alpha}^{\alpha'}\sum_{k,\ell}\int\wh{\rho}(t,\ell)\f{\ell_{\alpha-\alpha'}}{|\ell|^2}D^{\alpha'}_\eta\wh{f}(t,k-\ell,\eta-\ell\tent)A^\nu(t,k,\eta)\\
&\times k\cdot\bar{\eta}\ k_i^2A^\nu\ov{D^\alpha_\eta\wh{f}(t,k,\eta)}d\eta\ e^{-(2|\alpha|+1)\nu t}e^{2\delta_1\nu^\f13t} ,\\
&\sR^{ \alpha,i,hc}_{re}:=\smc\nu^\f43\sum_{\substack{\alpha'<\alpha\\|\alpha'|\geq|\alpha|-2}}C_{\alpha}^{\alpha'}\sum_k\int D^{\alpha-\alpha'}_\eta\rr{\bar{\eta}^2}A^\nu D^{\alpha'}_\eta\wh{f}(t,k,\eta)k\cdot\bar{\eta}\ k_i^2A^\nu\ov{D^\alpha_\eta\wh{f}(t,k,\eta)}d\eta\\
&\times e^{-2(|\alpha|+1)\nu t}e^{2\delta_1\nu^\f13t} .
\ea\]
We remark that the time derivative acting on $\vect$ yields $\sD^{ \alpha,i,hc}$ and $\sR^{ \alpha,i,hc}_{v}$, and the time derivative acting on $e^{-(|\alpha|+1)\nu t}$ yields $\sR^{ \alpha,i,hc}_{|\alpha|}$.
We shall  give the estimates term by term.
For $\sR^{ \alpha,i,hc}_{\delta_1}$, we note that
\[\ba|\sR^{ \alpha,i,hc}_{\delta_1}|\leq\f18\smc\nu^\f13\LLT{\na_x\pa_{x_i}A^\nu(v^\alpha f)(t)}^2e^{-2(|\alpha|+1)\nu t}e^{2\delta_1\nu^\f13t} +8\smc\nu\LLT{\vect\pa_{x_i}A^\nu(v^\alpha f)(t)}^2e^{-2(|\alpha|+1)\nu t}e^{2\delta_1\nu^\f13t} .\ea\]
We remark that the first term  and the second term can be absorbed by $\sD^{ \alpha,i,hc}$ and by $\sD^{ \alpha,i,x}$ respectively. $\sR^{ \alpha,i,hc}_{v}$ can be treated similarly to $\sR^{ \alpha,i,hc}_{\delta_1}$. Similarly,  $\sR^{ \alpha,i,hc}_{D}$ can be bounded by
\[\ba
|\sR^{ \alpha,i,hc}_{D}|\leq\f18\nu\LLT{\vect\na_x\pa_{x_i}A^\nu(v^\alpha f)(t)}^2e^{-2(|\alpha|+1)\nu t}e^{2\delta_1\nu^\f13t} +8\smc^2\nu^\f53\LLT{\rr{\vect}^2\pa_{x_i}A^\nu(v^\alpha f)(t)}^2e^{-2(|\alpha|+1)\nu t}e^{2\delta_1\nu^\f13t} .
\ea\]
The first term  and the second term can be absorbed by $\sD^{ \alpha,i,x}$ and by $\sD^{ \alpha,i,v}$. For $\sR^{ \alpha,i,hc}_{|\alpha|}$, we have
\[\ba
&|\sR^{ \alpha,i,hc}_{|\alpha|}|\leq8\smc\nu(|\alpha|+1)\LLT{\na_x\pa_{x_i}A^\nu(v^\alpha f)(t)}^2e^{-2(|\alpha|+1)\nu t}e^{2\delta_1\nu^\f13t} \\&+\f18\smc\nu^\f53(|\alpha|+1)\LLT{\vect\pa_{x_i}A^\nu(v^\alpha f)(t)}^2e^{-2(|\alpha|+1)\nu t}e^{2\delta_1\nu^\f13t} .
\ea\]
The first term  and the second term can be absorbed by $\sD^{ \alpha,i,x}_{|\alpha|}$ and  by $\sD^{ \alpha,i,v}_{|\alpha|}$. Similarly $\sR^{ \alpha,i,hc}_\lambda$ can be absorbed by $\sD^{ \alpha,i,x}_\lambda$ and $\sD^{ \alpha,i,v}_\lambda$.

$\bullet$ \underline{\emph{Estimate of $\sL^{ \alpha,i,hc}$}.} It holds that
\[\ba
&|\sL^{ \alpha,i,hc}|\ls\smc\nu^\f13\sum_k \int\nr{\wh{A^\nu\rho}(t,k)}\nr{D^\alpha_\eta\rr{\eta\wh{\mu}(\eta)}(\bar{\eta})\bar{\eta}A^\nu(t,0,\bar{\eta})}|k||k_i|\nr{A^\nu D^\alpha_\eta\wh{f}(t,k,\eta)}d\eta\ \emnt \\
&e^{-(|\alpha|+1)\nu t}e^{2\delta_1\nu^\f13t} 
\ls\smc\nu^\f13\LLT{A^\nu\rho(t)e^{\delta\nu^\f13t}\<\tent\>^3}\<t\>^{-3}\LLT{\na_x\pa_{x_i}A^\nu(v^\alpha f)(t)}e^{-(|\alpha|+1)\nu t}\\&e^{\delta_1\nu^\f13t}\ls{\smc}\nu^\f13\LLT{A^\nu\rho(t)e^{\delta\nu^\f13t}\<\tent\>^3}\<t\>^{-3}\eps,
\ea\]
which implies that $\int_0^{T^*} |\sL^{ \alpha,i,hc}|dt\ls\nu^{\f13}\eps^2$.

\underline{\emph{Estimate of $\sN^{ \alpha,i,hc}_{top}$}.} We first observe that 
\[\ba
&\sum_{k,\ell}\int\wh{\rho}(t,\ell)\f{\ell\cdot\bar{\eta}}{|\ell|^2}D^\alpha_\eta\wh{f}(t,k-\ell,\eta-\ell\tent)A^\nu(t,k,\eta)k\cdot\bar{\eta}\  k_i^2A^\nu\ov{D^\alpha_\eta\wh{f}(t,k,\eta)}d\eta\\
=&\sum_{k,\ell}\int\wh{\rho}(t,\ell)\f{\ell\cdot\bar{\eta}}{|\ell|^2}D^\alpha_\eta\wh{f}(t,k-\ell,\eta-\ell\tent)A^\nu(t,k,\eta)k\cdot\bar{\eta}\  \ell_ik_iA^\nu\ov{D^\alpha_\eta\wh{f}(t,k,\eta)}d\eta\\
&+\sum_{k,\ell}\int\wh{\rho}(t,\ell)\f{\ell\cdot\bar{\eta}}{|\ell|^2}(k_i-\ell_i)D^\alpha_\eta\wh{f}(t,k-\ell,\eta-\ell\tent)\\
&\times \rr{A^\nu(t,k,\eta)-A^\nu(t,k-\ell,\eta-\ell\tent)}k\cdot\bar{\eta}\ k_iA^\nu\ov{D^\alpha_\eta\wh{f}(t,k,\eta)}d\eta\\
&+\sum_{k,\ell}\int\wh{\rho}(t,\ell)\f{\ell\cdot\bar{\eta}}{|\ell|^2}(k_i-\ell_i)A^\nu D^\alpha_\eta\wh{f}(t,k-\ell,\eta-\ell\tent)k\cdot\bar{\eta}\ k_iA^\nu\ov{D^\alpha_\eta\wh{f}(t,k,\eta)}d\eta.
\ea\]
We denote the corresponding term as $\sN^{ \alpha,i,hc}_{top,1}, \sN^{ \alpha,i,hc}_{top,2}, \sN^{ \alpha,i,hc}_{top,3}$.  $\sN^{ \alpha,i,hc}_{top,1}$ can be bounded as follows:
\[\ba
&|\sN^{ \alpha,i,hc}_{top,1,LH}|\ls\smc\nu^\f13\sum_{k,\ell}\int\nr{\wh{A^\nu\rho}(t,\ell)}\<\ell\>^{-\f{d}{2}-\f12}\<\tent\>^{-5}e^{2\nu t}\<\tent\>^2\\
&\times \nr{\wh{A^\nu_2(v^\alpha f)}(t,k-\ell,\eta-\ell\tent)}|k||k_i|\nr{\wh{A^\nu(v^\alpha f)}(t,k,\eta)}d\eta\ e^{-2(|\alpha|+1)\nu t}e^{2\delta_1\nu^\f13t} \\
\ls&\smc\nu^\f13\LLTX{A^\nu\rho(t)e^{\delta\nu^\f13t}}\<t\>^{-3}\LLT{A^\nu_2(v^\alpha f)(t)}e^{-(|\alpha|+1)\nu t}\LLT{\na_x\pa_{x_i}A^\nu(v^\alpha f)(t)}\\
&\times e^{-(|\alpha|+1)\nu t}e^{\delta_1\nu^\f13t}\ls\smc\nu^\f13\LLTX{A^\nu\rho(t)e^{\delta\nu^\f13t}}\<t\>^{-3}\eps^2,
\ea\]
and 
\[\ba
&|\sN^{ \alpha,i,hc}_{top,1,HL}|\ls\smc\nu^\f13\sum_{k,\ell}\int\nr{\wh{A^\nu\rho}(t,\ell)}e^{2\nu t}\<t\>^2\nr{\wh{A^\nu(v^\alpha f)}(t,k-\ell,\eta-\ell\tent)}\<k-\ell\>^{-d}\\
&\times |k||k_i|\nr{\wh{A^\nu(v^\alpha f)}(t,k,\eta)}d\eta\ e^{-2(|\alpha|+1)\nu t}e^{2\delta_1\nu^\f13t} \\
\ls&\smc\nu^\f13\LLTX{A^\nu\rho(t)e^{\delta\nu^\f13t}\<\tent\>^3}\<t\>^{-1}\LLT{A^\nu(v^\alpha f)(t)}e^{-(|\alpha|+1)\nu t}\LLT{\na_x\pa_{x_i}A^\nu(v^\alpha f)(t)}\\
\times &e^{-(|\alpha|+1)\nu t}e^{\delta_1\nu^\f13t}\ls\smc\nu^\f13\eps^2\LLTX{A^\nu\rho(t)e^{\delta\nu^\f13t}\<\tent\>^3}\<t\>^{-1}.
\ea\]
Next we estimate $\sN^{ \alpha,i,hc}_{top,2}$. For the LH part,  using \eqref{commu}, we have
\[\ba
&|\sN^{ \alpha,i,hc}_{top,2,LH}|\ls\smc\nu^\f13\sum_{k,\ell}\int\nr{\wh{A^\nu\rho}(t,\ell)}\<\ell\>^{-\f{d}{2}-\f32}\<\tent\>^{-5}\ent\<\tent\>\<\ell,\ell\tent\>|\bar{\eta}||k_i-\ell_i|\\
&\times \<k-\ell,\eta-\ell\tent\>^{\f s2}\nr{\wh{A^\nu(v^\alpha f)}(t,k-\ell,\eta-\ell\tent)}|k||k_i|\<k,\eta\>^{\f s2}\nr{\wh{A^\nu(v^\alpha f)}(t,k,\eta)}d\eta\\&\times  e^{-2(|\alpha|+1)\nu t}e^{2\delta_1\nu^\f13t} \ls\smc\nu^\f13\<\tent\>^{-3}\emnt\LLT{\<v\>^mA^\nu f_{\neq}(t)e^{\f{\delta_1}{2}\nu^\f13 t}}\\
&\times \LLT{\vect\pa_{x_i}A^\nu_{\f s2}(v^\alpha f)(t)}\LLT{\na_x\pa_{x_i}A^\nu_{\f s2}(v^\alpha f)(t)}e^{-2(|\alpha|+1)\nu t}e^{2\delta_1\nu^\f13t} \\
\ls&\eps\<\tent\>^{-3}\emnt\LLT{\na_x\pa_{x_i}A^\nu_{\f s2}(v^\alpha f)(t)}^2e^{-2(|\alpha|+1)\nu t}e^{2\delta_1\nu^\f13t} \\
&+\eps\smc^2\nu^\f23\<\tent\>^{-3}\emnt\LLT{\vect\pa_{x_i}A^\nu_{\f s2}(v^\alpha f)(t)}^2e^{-2(|\alpha|+1)\nu t}e^{2\delta_1\nu^\f13t} .
\ea\]
The first term and the second term can be absorbed by $\sD^{ \alpha,i,x}_\lambda$ and by $\sD^{ \alpha,i,v}_\lambda$. For the HL part, we get that
\[\ba
&|\sN^{ \alpha,i,hc}_{top,2,HL}|\ls\smc\nu^\f13\LLTX{A^\nu\rho(t)e^{\delta\nu^\f13t}\<\tent\>^4}\<t\>^{-1}\LLT{A^\nu(v^\alpha f)(t)}e^{-(|\alpha|+1)\nu t}\LLT{\na_x\pa_{x_i}A^\nu(v^\alpha f)(t)}\\
&\times e^{-(|\alpha|+1)\nu t}e^{\delta_1\nu^\f13t}\ls\smc\nu^\f13\eps^2\LLTX{A^\nu\rho(t)e^{\delta\nu^\f13t}\<\tent\>^4}\<t\>^{-1}.
\ea\]
Next we estimate $\sN^{ \alpha,i,hc}_{top,3}$. Since
\[\ba
&\Re\sum_{k,\ell}\int\wh{\rho}(t,\ell)\f{\ell\cdot\bar{\eta}}{|\ell|^2}(k_i-\ell_i)A^\nu D^\alpha_\eta\wh{f}(t,k-\ell,\eta-\ell\tent)k\cdot\bar{\eta}\ k_iA^\nu\ov{D^\alpha_\eta\wh{f}(t,k,\eta)}d\eta\\
=&-\Re\sum_{k,\ell}\int\wh{\rho}(t,\ell)\f{\ell\cdot\bar{\eta}}{|\ell|^2}(k_i-\ell_i)A^\nu D^\alpha_\eta\wh{f}(t,k-\ell,\eta-\ell\tent)(k-\ell)\cdot\bar{\eta}\ k_iA^\nu\ov{D^\alpha_\eta\wh{f}(t,k,\eta)}d\eta,
\ea\]
we have
\[\ba
&\Re\sum_{k,\ell}\int\wh{\rho}(t,\ell)\f{\ell\cdot\bar{\eta}}{|\ell|^2}(k_i-\ell_i)A^\nu D^\alpha_\eta\wh{f}(t,k-\ell,\eta-\ell\tent)k\cdot\bar{\eta}\ k_iA^\nu\ov{D^\alpha_\eta\wh{f}(t,k,\eta)}d\eta\\
=&\Re\sum_{k,\ell}\int\wh{\rho}(t,\ell)\f{\ell\cdot\bar{\eta}}{|\ell|^2}(k_i-\ell_i)A^\nu D^\alpha_\eta\wh{f}(t,k-\ell,\eta-\ell\tent)\ell\cdot\bar{\eta}\ k_iA^\nu\ov{D^\alpha_\eta\wh{f}(t,k,\eta)}d\eta\\
&+\Re\sum_{k,\ell}\int\wh{\rho}(t,\ell)\f{\ell\cdot\bar{\eta}}{|\ell|^2}(k_i-\ell_i)A^\nu D^\alpha_\eta\wh{f}(t,k-\ell,\eta-\ell\tent)(k-\ell)\cdot\bar{\eta}\ k_iA^\nu\ov{D^\alpha_\eta\wh{f}(t,k,\eta)}d\eta\\
=&\f12\Re\sum_{k,\ell}\int\wh{\rho}(t,\ell)\f{\ell\cdot\bar{\eta}}{|\ell|^2}(k_i-\ell_i)A^\nu D^\alpha_\eta\wh{f}(t,k-\ell,\eta-\ell\tent)\ell\cdot\bar{\eta}\ k_iA^\nu\ov{D^\alpha_\eta\wh{f}(t,k,\eta)}d\eta.
\ea\]
Thus
\[\ba
&|\sN^{ \alpha,i,hc}_{top,3}|\ls\smc\nu^\f13\sum_{k,\ell}\int\nr{\wh{A^\nu\rho}(t,\ell)}\<\ell\>^{-\f{d}{2}-\f12}\<\tent\>^{-4}e^{\nu t}\<\tent\>\\
&\times \nr{\wh{A^\nu_2(v^\alpha f)}(t,k-\ell,\eta-\ell\tent)}|\bar{\eta}||k_i|\nr{\wh{A^\nu(v^\alpha f)}(t,k,\eta)}d\eta\ e^{-2(|\alpha|+1)\nu t}e^{2\delta_1\nu^\f13t} \\
\ls&\smc\nu^\f13\LLTX{A^\nu\rho(t)e^{\delta\nu^\f13t}}\<t\>^{-3}\LLT{A^\nu_2(v^\alpha f)(t)}e^{-(|\alpha|+1)\nu t}\LLT{\vect\pa_{x_i}A^\nu(v^\alpha f)(t)}\\
&\times e^{-(|\alpha|+1)\nu t}e^{\delta_1\nu^\f13t}\ls\sqrt{\smc}\eps^2\LLTX{A^\nu\rho(t)e^{\delta\nu^\f13t}}\<t\>^{-3}.
\ea\]

$\bullet$\underline{\emph{Estimate of $\sN^{ \alpha,i,hc}_{re}$} and $\sR^{ \alpha,i,hc}_{re}$.} We have
\[\ba
&|\sN^{ \alpha,i,hc}_{re}|\ls\smc\nu^\f13\sum_{\substack{\alpha'<\alpha\\|\alpha'|=|\alpha|-1}}C_{\alpha}^{\alpha'}\left(\LLTX{A^\nu\rho(t)e^{\delta\nu^\f13t}}\<t\>^{-2}\LLT{A^\nu_2(v^{\alpha'} f)(t)}\right.\\
&\times e^{-(|\alpha'|+1)\nu t}\LLT{\na_x\pa_{x_i}A^\nu(v^\alpha f)(t)}e^{-(|\alpha|+1)\nu t}e^{\delta_1\nu^\f13t}+\LLTX{A^\nu\rho(t)e^{\delta\nu^\f13t}\<\tent\>^4}\<t\>^{-1}\\
&\left.\times \LLT{A^\nu(v^{\alpha'} f)(t)}e^{-(|\alpha'|+1)\nu t}\LLT{\na_x\pa_{x_i}A^\nu(v^\alpha f)(t)}e^{-(|\alpha|+1)\nu t}e^{\delta_1\nu^\f13t}\right)\\
&\ls\eps^2\LLTX{A^\nu\rho(t)e^{\delta\nu^\f13t}\<\tent\>^3}\<t\>^{-1},
\ea\]
and
\[\ba
&|\sR^{ \alpha,i,hc}_{re}|\leq\f14\smc\nu^\f53\LLT{\rr{\vect}^2\pa_{x_i}A^\nu(v^\alpha f)(t)}^2e^{-2(|\alpha|+1)\nu t}e^{2\delta_1\nu^\f13t} +\f14\smc\nu^\f53\LLT{\vect\pa_{x_i}A^\nu(v^\alpha f)(t)}^2\\
&\times e^{-2(|\alpha|+1)\nu t}e^{2\delta_1\nu^\f13t} +C_m\smc\nu\sum_{\substack{\alpha'<\alpha\\|\alpha'|\geq|\alpha|-2}}\LLT{\na_x\pa_{x_i}A^\nu(v^{\alpha'} f)(t)}^2e^{-2(|\alpha'|+1)\nu t}e^{2\delta_1\nu^\f13t} .
\ea\]
The first term and the second term can be absorbed by $\sD^{ \alpha,i,v}$ and by $\sD^{ \alpha,i,v}_{|\alpha|}$. The third term can be absorbed by $\sD^{ \alpha',i,x}_{|\alpha'|}$ for $ \alpha'<\alpha $ and $|\alpha'|\geq|\alpha|-2$.

\bigskip
In conclusion, we get that
\ben\label{improvehED}
\sup_{t\in [0,T^*]}E^{ED}(t)\leq E^{ED}(0)+\int_0^{T^*}\nr{\f{d}{dt}E^{ED}(t)}dt\leq C_{6}(m,d)\eps^2\notag\\+C_7(K_{\rho })\eps^2+C_{8}(K_{\rho },K_{f},K_{ED },K_{SH})\eps^3.
\een
We remark that the first term comes from the initial data and the second term comes from the linear term.

\section{Spatially homogeneous thermalization estimate}
We set $Lh:=\Delta_vh+\na_v\cdot(vh)$ and recall the following  estimate of the semi-group(see \cite{Bj} Theorem 2):
\begin{prop}
	Let $m>\f d2+1$ be an integer, $\mathsf{s}\geq0$. There exists $C_{\mathsf{s},m}$ such that for any mean-zero $g(v)$,
	\[\nnr{e^{\nu tL}g}_{H^\mathsf{s}_m}\ls\emnt\nnr{g}_{H^\mathsf{s}_m}.\]
\end{prop}

We may check that $h_0(t)$ satisfies the equation
\[\pa_t h_0(t)+\rr{E\cdot\na_vh}_0=\nu Lh_0.\]
By Duhamel's principle, it holds that
\[h_0(t)=e^{\nu tL}{h_{in}}_0-\int_0^te^{\nu (t-\tau)L}\rr{E\cdot\na_vh(\tau)}_0d\tau=\bI+\bN.\]

$\bullet$ \underline{\emph{Estimate of $\bI$}.}  It holds that
\[\nnr{A^\nu_{-\beta}(t,\emnt\na)e^{\nu tL}{h_{in}}_0}_{H^{\beta+1}_m}=\nnr{e^{\nu tL}A^\nu_{-\beta}(t,\na){h_{in}}_0}_{H^{\beta+1}_m}\ls\emnt\nnr{A^\nu_{-\beta}(t,\na){h_{in}}_0}_{H^{\beta+1}_m}\ls\eps\emnt. \]

$\bullet$ \underline{\emph{Estimate of $\bN$}.} We observe that 
\[\ba
&\int_0^t\nnr{A^\nu_{-\beta}(t,\emnt\na)e^{\nu (t-\tau)L}\rr{E\cdot\na_vh(\tau)}_0}_{H^{\beta+1}_m}d\tau=\int_0^t\nnr{e^{\nu (t-\tau)L}A^\nu_{-\beta}(t,\emntau\na)\rr{E\cdot\na_vh(\tau)}_0}_{H^{\beta+1}_m}d\tau\\
\ls&\int_0^te^{-\nu(t-\tau)}\nnr{A^\nu_{-\beta}(t,\emntau\na)\rr{E\cdot\na_vh(\tau)}_0}_{H^{\beta+1}_m}d\tau\ls\emnt\int^t_0\entau\nnr{A^\nu_{-\beta}(\tau,\emntau\na)\rr{E\cdot\na_vh(\tau)}_0}_{H^{\beta+1}_m}d\tau.
\ea\]

We further have
\[\ba
&\entau\nnr{A^\nu_{-\beta}(\tau,\emntau\na)\rr{E\cdot\na_vh(\tau)}_0}_{H^{\beta+1}_m}\\
\ls&\sum_{|\alpha|\leq m}\entau\nnr{\<\eta\>^{\beta+1}A^\nu_{-\beta}(\tau,\emntau\eta)\rr{\sum_\ell\wh{\rho}(\tau,\ell)\f{\ell}{|\ell|^2}\cdot D^\alpha_\eta\rr{\eta\wh{f}(\tau,-\ell,\emntau\eta-\ell\tentau)}   }}_{L^2_\eta}\\
\ls&\sum_{|\alpha|\leq m}e^{(\beta+2)\nu\tau}\nnr{A^\nu_{1}(\tau,\emntau\eta)\rr{\sum_\ell\wh{\rho}(\tau,\ell)\f{\ell\cdot\eta}{|\ell|^2}D^\alpha_\eta\wh{f}(\tau,-\ell,\emntau\eta-\ell\tentau)   }}_{L^2_\eta}\\
&+\sum_{|\alpha|\leq m}e^{(\beta+2)\nu\tau}\sum_{\substack{\alpha'<\alpha\\|\alpha'|=|\alpha|-1}}C_\alpha^{\alpha'}\nnr{A^\nu_{1}(\tau,\emntau\eta)\rr{\sum_\ell\wh{\rho}(\tau,\ell)\f{\ell_{\alpha-\alpha'}}{|\ell|^2}D^{\alpha'}_\eta\wh{f}(\tau,-\ell,\emntau\eta-\ell\tentau)   }}_{L^2_\eta}\\
=&\bN_{top}^h+\bN_{re}^h.
\ea\]
We notice that $(0,\emntau\eta)=(\ell,\ell\tentau)+(-\ell,\eta-\ell\tentau)$ and $|\eta|\leq\entau\<\tentau\>\<-\ell,\eta-\ell\tentau\>$. 
\[\ba
&\bN_{top}^h\ls\sum_{|\alpha|\leq m}e^{(\beta+2)\nu\tau}\nnr{\sum_\ell\nr{\wh{A^\nu_{\f12}\rho}(\tau,\ell)}|\ell|^\f12\<\tau\>^{\f32}\entau\nr{A^\nu_2D^\alpha_\eta\wh{f}(\tau,-\ell,\emntau\eta-\ell\tentau)}   }_{L^2_\eta}\\
\ls&\sum_{|\alpha|\leq m}e^{(\beta+3+\f d2)\nu\tau}\sum_\ell\nr{\wh{A^\nu_{\f12}\rho}(\tau,\ell)}|\ell|^\f12\<\tau\>^\f32\nnr{A^\nu_2D^\alpha_\eta\wh{f}(\tau,-\ell,\eta)}_{L^2_\eta}\\
\ls&\sum_{|\alpha|\leq m}\LLTX{|\na|^\f12A^\nu_{\f12}\rho(\tau)e^{\delta\nu^\f13\tau}\<\tentau\>^{\f52}}\<\tau\>^{-1}e^{-(m+1)\nu\tau}\LLT{A^\nu_2(v^\alpha f)(\tau)}\\
\ls_m& \eps\LLTX{|\na|^\f12A^\nu_{\f12}\rho(\tau)e^{\delta\nu^\f13\tau}\<\tentau\>^\f52}\<\tau\>^{-1}.
\ea\]
Thus
\[\int_0^t\nr{\bN_{top}^h(\tau)}d\tau\ls\eps\nnr{|\na|^\f12A^\nu_{\f12}\rho(t)e^{\delta\nu^\f13t}\<\tentau\>^{\f52}}_{L^2([0,T^*],L^2_x)}\ls\eps^2.\]

$\bN_{re}^h$ can be treated in the same manner. 

\bigskip
In conclusion, we get that
\ben\label{improveHsh}
\nnr{A^\nu_{-\beta}(t,0,\emnt\na)h_0(t)}_{H^{\beta+1}_m}\ent\leq C_{9}(m,d)\eps+C_{10}(K_{\rho },K_{f},K_{SH},K_{ED})\eps^2.
\een

\section{Proof of Theorem \ref{T1}}
The section is to give a detailed proof for the main theorem.

\begin{proof}[Proof of Theorem \ref{T1}] We split the proof into two steps.

\underline{{\it Step 1: Global stability}.} Under the ansatz (\ref{Hrho}-\ref{Hsh}), we improve them in  Section 2-Section 5 thanks to  \eqref{improveHrho}, \eqref{improvef}, \eqref{improvehED}, \eqref{improveHsh}. These enable  us to use continuity argument to get the global stability.

 For $s\geq\f13$, we choose $K_{\rho}=2C_1(m,d)$ and $\eps$ sufficiently small such that $C_2(K_{\rho}, K_f, K_{ED}, K_{SH})\eps<C_1(m,d)$. For $s<\f13$,  we choose $K_{\rho }=2C_1(m,d)$ and $c$ sufficiently small such that $C_2(K_{\rho}, K_f, K_{ED}, K_{SH})c<C_1(m,d)$. Thanks to \eqref{improveHrho} we get that
	\[\LLTXT{|\na_x|^\f12A^\nu_\f12\rho(t)e^{\delta\nu^\f13t}}\leq K_{\rho}\eps.\]
	
By properly choosing $K_1<K_2<K_3<\cdots<K_m$ (shown in Section 3 and Section 4), we get \eqref{improvef} in Section 3. If we choose $K_{f}=\rr{2C_3(m,d)+2C_4(K_{\rho })}^\f12$ and $\eps$ sufficiently small such that $C_5(K_{\rho}, K_f, K_{ED}, K_{SH})\eps<K_f^2/4$, then it holds that
\[\sum_{|\alpha|\leq m}\f{e^{-2(|\alpha|+1)\nu t}}{K_{|\alpha|}}\LLT{A^\nu_2(v^\alpha f)(t)}^2\leq (K_{f})^2\eps^2.\]
If we choose $K_{ED}=\rr{2C_6(m,d)+2C_7(K_{\rho })}^\f12$ and $\eps$ sufficiently small such that $C_8(K_{\rho}, K_f, K_{ED}, K_{SH})\eps<K_{ED}^2/4$, then it holds that
\[\sum_{|\alpha|\leq m}\f{e^{-2(|\alpha|+1)\nu t}}{K_{|\alpha|}}\left(\LLT{\na_x^2A^\nu(v^\alpha f)(t)}^2+\smc\nu^\f23\LLT{\vect \na_xA^\nu(v^\alpha f)(t)}^2\right)\leq (K_{ED})^2\eps^2e^{-2\delta_1\nu^\f13t}.\]
If we choose $K_{SH}=2C_9(m,d)$ and $\eps$ sufficiently small such that $C_{10}(K_{\rho}, K_f, K_{ED}, K_{SH})\eps<K_{SH}/4$, then it holds that
\[\nnr{A^\nu_{-\beta}(t,0,\emnt\na)h_0(t)}_{H^{\beta+1}_m}\leq K_{SH}\eps\emnt.\]

These in particular imply that the bootstrap hypotheses hold for all the time. Thus we get the global stability of the solution.

\underline{{\it Step 2: Asymptotic behavior}.} We first prove the enhanced dissipation estimate. By Sobolev embedding, it holds that
	\[\LLT{\<v\>^m e^{\lambda_\infty\<\na_{x,v}\>^s}f_{\neq}(t)e^{\delta_2\nu^\f13t}}^2\ls_{m,\delta_1,\delta_2} \sum_{|\alpha|\leq m}\f{e^{-2(|\alpha|+1)\nu t}}{K_{|\alpha|}}\LLT{\na_x^2A^\nu(v^\alpha f)(t)e^{\delta_1\nu^\f13t}}^2\leq (2K_{ED})^2\eps^2,\]
	where we use \eqref{Hed}. Next we prove the spatially homogeneous thermalization estimate. We have
	\[\nnr{e^{\lambda_\infty\<\emnt\na_v\>^s}h_0(t)}_{L^2_m}\ls \nnr{A^\nu_{-\beta}(t,0,\emnt\na)h_0(t)}_{H^{3}_m}\leq 2K_{SH}\eps e^{-\nu t}.\] 
	For the density estimate, we observe that
	\[\nr{\wh{\rho}(t,k)}\ls \LLT{\<v\>^mA^\nu f_{\neq}(t)e^{\f{\delta_1}{2}\nu^\f13t}}e^{-\f{\delta_1}{2}\nu^\f13t}e^{-\lambda_\infty\<k,k\tent\>}\ls_N 2K_{ED}\eps \<k\>^{-N}\<t\>^{-N},\quad \forall N>0.\]
	Here we use \eqref{Hed} and Proposition \ref{nut} and then complete the proof.
\end{proof}

\bigskip

 {\bf Acknowledgments.}   Yue Luo would like to thank his adviser Professor Ling-Bing He for the profitable discussion.

\end{document}